\documentclass[reqno]{amsart}
\usepackage{amsthm,amsmath,amssymb}
\usepackage{stmaryrd,mathrsfs}
\usepackage[T1]{fontenc}
\usepackage{esint}
\usepackage{tikz}

\allowdisplaybreaks

\newcommand{\ud}[0]{\,\mathrm{d}}

\newcommand{\abs}[1]{|#1|}

\newcommand{\Norm}[2]{\|#1\|_{#2}}

\newcommand{\BNorm}[2]{\Big\|#1\Big\|_{#2}}
\newcommand{\pair}[2]{\langle #1,#2 \rangle}

\newcommand{\ave}[1]{\langle #1\rangle}

\newcommand{\R}{\mathbb{R}}

\swapnumbers \numberwithin{equation}{section}

\newcommand{\cf}{\mathcal{F}}
\newcommand{\dsigma}{\mathrm{d}\sigma}
\newcommand{\br}{\mathbb{R}}
\newcommand{\cd}{\mathcal{D}}
\newcommand{\Nnorm}[1]{\|#1\|}
\newcommand{\ch}{\textup{ch}}

\theoremstyle{plain}
\newtheorem{theorem}[equation]{Theorem}
\newtheorem{proposition}[equation]{Proposition}
\newtheorem{corollary}[equation]{Corollary}
\newtheorem{lemma}[equation]{Lemma}

\theoremstyle{definition}
\newtheorem{definition}[equation]{Definition}

\theoremstyle{remark}

\newtheorem{example}[equation]{Example}

\makeatletter
\@namedef{subjclassname@2010}{%
  \textup{2010} Mathematics Subject Classification}
\makeatother

\begin{document}

\title[{Two-weight $L^p$-$L^q$ bounds for $p\leq q$ and $p>q$}]{Two-weight $L^p$-$L^q$ bounds for positive dyadic operators: unified approach to $p\leq q$ and $p>q$}

\author{Timo S. H\"anninen \and Tuomas P. Hyt\"onen}
\address{T.S.H. and T.P.H: Department of Mathematics and Statistics, P.O.B.~68 (Gustaf H\"all\-str\"omin katu~2b), FI-00014 University of Helsinki, Finland}
\email{timo.s.hanninen@helsinki.fi}
\email{tuomas.hytonen@helsinki.fi}


\author{Kangwei Li}
\address{K.L.: School of Mathematical Sciences and LPMC,  Nankai University,
      Tianjin~300071, China}
\email{likangwei9@mail.nankai.edu.cn}

%
%


\thanks{T.S.H. and T.P.H. are supported by the European Union through the ERC Starting Grant ``Analytic-probabilistic methods for borderline singular integrals'', and are members of the Finnish Centre of Excellence in Analysis and Dynamics Research. K.L. is partially supported  by the
National Natural Science Foundation of China (11371200), the Research Fund for the Doctoral Program
of Higher Education (20120031110023) and the Ph.D. Candidate Research Innovation Fund of Nankai University. This research was conducted during K.L.'s visit to the University of Helsinki.}

\keywords{}
\subjclass[2010]{42B25, 47G40}


\begin{abstract}
We characterize the $L^p(\sigma)\to L^q(\omega)$ boundedness of positive dyadic operators of the form
\begin{equation*}
  T(f\sigma)=\sum_{Q\in\mathscr{D}}\lambda_Q\int_Q f\ud\sigma\cdot 1_Q,
\end{equation*}
and the $L^{p_1}(\sigma_1)\times L^{p_2}(\sigma_2)\to L^q(\omega)$ boundedness of their bilinear analogues, for arbitrary locally finite measures $\sigma,\sigma_1,\sigma_2,\omega$. In the linear case, we unify the existing ``Sawyer testing'' (for $p\leq q$) and ``Wolff potential'' (for $p>q$) characterizations into a new ``sequential testing'' characterization valid in all cases. We extend these ideas to the bilinear case, obtaining both sequential testing and potential type characterizations for the bilinear operator and all $p_1,p_2,q\in(1,\infty)$. Our characterization covers the previously unknown case $q<p_1p_2/(p_1+p_2)$, where we introduce a new two-measure Wolff potential.
\end{abstract}

\maketitle
\tableofcontents

\section{Introduction}

%
%

We are interested in characterizations of two-weight $L^p(\sigma)$-$L^q(\omega)$ inequalities for discrete positive operators of the form
\begin{equation}\label{eq:TandTQ}
\begin{split}
  T(f\sigma) &:=\sum_{Q\in\mathcal{D}}\lambda_Q\int_Q f \ud\sigma\cdot 1_Q,\qquad\lambda_Q\geq 0,\\
  T_{Q}(f\sigma) &:=\sum_{\substack{Q'\in\mathcal{D}\\ Q'\subseteq Q}}\lambda_{Q'}\int_{Q'} f \ud\sigma\cdot 1_{Q'},
\end{split}
\end{equation}
as well as some bilinear generalizations. 
Such questions have been quite extensively studied in the last few years, but our goal is to offer a systematic approach, which unifies a number of existing results in the linear case, and more or less completes the picture in the bilinear case.


For the $L^p(\sigma)$-$L^q(\omega)$ boundedness of $T$ from \eqref{eq:TandTQ}, there are two seemingly quite different characterizations depending on the relative size of the exponents $p$ and $q$.

\begin{theorem}[Linear $p\leq q$; Lacey, Sawyer, and Uriarte-Tuero \cite{lacey2009}]\label{thm:LSU}
For $1<p\leq q<\infty$ and two measures $\sigma,\omega$, we have
\begin{equation*}
  \Norm{T(\,\cdot\,\sigma)}{L^p(\sigma)\to L^q(\omega)}
  \eqsim \mathfrak{T}+\mathfrak{T}^*,
\end{equation*}
where
\begin{equation}\label{eq:SawyerTesting}
  \mathfrak{T}  :=\sup_{Q\in\mathcal{D}}\frac{\Norm{T_{Q}(\sigma)}{L^q(\omega)}}{\sigma(Q)^{1/p}}, \qquad
  \mathfrak{T}^*  :=\sup_{Q\in\mathcal{D}}\frac{\Norm{T_{Q}(\omega)}{L^{p'}(\sigma)}}{\omega(Q)^{1/q'}}.
\end{equation}
\end{theorem}

\begin{theorem}[Linear $p> q$; Tanaka \cite{tanaka2014}]\label{thm:Tanaka}
For $1<q< p<\infty$ and two measures $\sigma,\omega$, we have
\begin{equation*}
  \Norm{T(\,\cdot\,\sigma)}{L^p(\sigma)\to L^q(\omega)}
  \eqsim \Norm{W^q_{\lambda,\omega}[\sigma]}{L^{r/q}(\sigma)}^{1/q}+\Norm{W^{p'}_{\lambda,\sigma}[\omega]}{L^{r/p'}(\omega)}^{1/p'}
\end{equation*}
where $1/r=1/q-1/p$ and
\begin{equation*}
  W^q_{\lambda,\omega}[\sigma]
  :=\sum_{Q\in\mathcal{D}}\lambda_Q \omega(Q)\Big(\frac{1}{\omega(Q)}\sum_{Q'\subseteq Q}\lambda_{Q'}\sigma(Q')\omega(Q')\Big)^{q-1}1_Q
\end{equation*}
is the \emph{discrete Wolff potential}.
\end{theorem}

The case $p=q=2$ of Theorem~\ref{thm:LSU} was already proven by Nazarov, Treil and Volberg \cite{NTV:1999}. The case $p>q$ (as in Theorem~\ref{thm:Tanaka}) was previously characterized by Cascante, Ortega and Verbitsky~\cite{cov2006}, even for $0<q<\infty$, but only under additional \emph{a priori} conditions (so called ``dyadic logarithmic bounded oscillation'') on the coefficients $\lambda_Q$; but our emphasis is on results valid without any such assumptions.

While both Theorem~\ref{thm:LSU} and Theorem~\ref{thm:Tanaka} characterize a two-weight inequality for an operator in terms of something arguably simpler, the characterizations are very different. The Sawyer-type testing conditions of Theorem~\ref{thm:LSU} are obviously necessary: they amount to the trivial bound $T_Q(\sigma)\leq T(1_Q\sigma)$ and the boundedness of $T(\,\cdot\,\sigma)$ and its formal adjoint $T(\,\cdot\,\omega)$ on the indicator functions, where the latter does not make any use of the properties of the operator other than boundedness and linearity (for the existence of an adjoint).

In a sense, the characterization of Theorem~\ref{thm:Tanaka} is even better, since it reduces the norm of an operator to the norm of a function, which in principle is a much simpler object. On the other hand, the relation of this function to the original operator is not at all obvious at first sight, and both the necessity and the sufficiency parts of Theorem~\ref{thm:Tanaka} are non-trivial results, which use the specific structure of the operator in question. Indeed, the very formula of $W^q_{\lambda,\omega}[\sigma]$ makes explicit reference to this structure, and does not suggest an immediate generalization to more general operators. Also,  there is a discontinuity between the characterizations of Theorems \ref{thm:LSU} and \ref{thm:Tanaka}, in that the Wolff potential conditions do not seem to converge to the Sawyer testing conditions in any obvious sense as $q\to p$.

Our first contribution is to provide a characterization that fixes this discontinuity and covers all values of $p,q\in(1,\infty)$ in a unified manner. This is achieved by a new ``sequential'' testing conditions, as defined below: (An alternative form of a unified characterization has been simultaneously obtained by Vuorinen \cite{Vuorinen}.)

\begin{theorem}\label{thm:seqTesting}
For $p,q\in(1,\infty)$ and two measures $\sigma,\omega$, we have
\begin{equation*}
  \Norm{T(\,\cdot\,\sigma)}{L^p(\sigma)\to L^q(\omega)}
  \eqsim \mathfrak{T}_r+\mathfrak{T}_r^*,
\end{equation*}
where $r\in(1,\infty]$ is determined by
\begin{equation}\label{eq:rDef}
  \frac{1}{r}=\Big(\frac{1}{q}-\frac{1}{p}\Big)_+
  :=\max\Big\{\frac{1}{q}-\frac{1}{p},0\Big\}.
\end{equation}
and
\begin{equation}\label{eq:newTesting}
\begin{split}
  \mathfrak{T}_r & :=\sup_{\cf}\BNorm{\Big\{\frac{\Norm{T_{F}(\sigma)}{L^q(\omega)}}{\sigma(F)^{1/p}}\Big\}_{F\in\cf}}{\ell^r}, \\
  \mathfrak{T}_r^* & :=\sup_{\mathcal{G}}\BNorm{\Big\{\frac{\Norm{T_{G}(\omega)}{L^{p'}(\sigma)}}{\omega(G)^{1/q'}}\Big\}_{G\in\mathcal{G}}}{\ell^r},
\end{split}
\end{equation}
where the supremums are taken over all subcollections $\cf$ and $\mathcal{G}$ of $\mathcal{D}$ that are sparse (in the sense of Definition~\ref{def:sparse} below) with respect to $\sigma$ and $\omega$, respectively.
\end{theorem}

\begin{definition}\label{def:sparse}
A family $\cf\subseteq\mathcal{D}$ is called $\sigma$-sparse (or sparse with respect to $\sigma$), if for every $F\in\cf$, there exists some $E(F)\subseteq F$ such that $\sigma(E(F))\geq\frac12\sigma(F)$, and the sets $E(F)$, $F\in\cf$, are pairwise disjoint.
\end{definition}

We note that a version of the sequential testing has been earlier used by Sawyer \cite[Theorem 3]{sawyer1984} to give the following characterization of the two-weight inequality for the Hardy operator $H(f\sigma)(x):=\int_0^x f(t)\dsigma(t)$:
\begin{equation*}
  \Norm{H(\,\cdot\,\sigma)}{L^p(\sigma)\to L^q(\omega)}
  \eqsim \sup_{\cdots<x_k<x_{k+1}\cdots}\lVert\{\sigma((x_{k-1},x_k])^{1/p'}\omega([x_k,x_{k+1}))^{1/q}\}_{x_k}\rVert_{\ell^r}.
\end{equation*}
However, it seems that this idea has not been previously pushed to more complicated positive integral operators, of which the Hardy operator is the simplest prototype case. (Sawyer's theorem is also valid for $(p,q)\in[1,\infty)\times(0,\infty)$, a larger range than in our result.)

When $p\leq q$, we have $r=\infty$ in \eqref{eq:rDef}, and the characterization \eqref{eq:newTesting} reduces to the Sawyer conditions \eqref{eq:SawyerTesting}.
For $p>q$, the necessity of the finiteness of $\mathfrak{T}_r,\mathfrak{T}^*_r$ from \eqref{eq:newTesting} is perhaps still not obvious. Nevertheless, we shall show that it is an essentially more general property, in that replacing each $T_F(\sigma)$ by the larger function $T(1_F\sigma)$ in $\mathfrak{T}_r$, we obtain a condition that is necessary for the boundedness of any positive linear $T(\,\cdot\,\sigma):L^p(\sigma)\to L^q(\omega)$, not involving any specific structure of the operator.

(Note that we have tagged the label `$\sigma$' with `$T(\,\cdot\,\sigma)$'.  In the case of a general operator, this tag just emphasizes that the operator $T(\,\cdot\,\sigma)$ acts on $L^p(\sigma)$. In the case of an integral operator associated with a kernel $K$, this tag indicates the measure with respect to which the integration is done: $T(f\sigma)(x):=\int K(x,y)f(y)\dsigma(y)$.)

We also demonstrate the necessity of the sequential testing conditions in another sense; namely, replacing them by the Sawyer testing conditions (i.e., replacing $\ell^r$ by the weaker norm $\ell^\infty$) is not in general sufficient for the boundedness when $p>q$.

By going through both Theorems \ref{thm:Tanaka} and \ref{thm:seqTesting}, one sees that the sequential testing and Wolff potential conditions are equivalent. However, this can also be seen more directly, and on a more general level: it turns out that the sequential testing is equivalent to a norm bound for an abstract ``Wolff potential'' for a general positive linear~$T(\,\cdot\,\sigma)$:

\begin{proposition}\label{prop:abstractWolff}
Let $1<q<p<\infty$ and $r\in(1,\infty)$ be given by $1/r=1/q-1/p$. Let $T(\,\cdot\,\sigma):L^p(\sigma)\to L^q(\omega)$ be a positive linear operator, and let $T_Q(\,\cdot\,\sigma)$ denote any of the following localized operators:
\begin{enumerate}
  \item\label{it:globTest} $T_Q(\,\cdot\,\sigma)=T(\cdot1_Q\sigma)$, or
  \item\label{it:locTest} $T_Q(\,\cdot\,\sigma)=1_Q T(\cdot1_Q\sigma)$, or
  \item\label{it:specTest} $T_Q(\,\cdot\,\sigma)$ be as in \eqref{eq:TandTQ} if $T$ of the form given in the same formula,
\end{enumerate}
and let $\mathfrak{T}_r$ be as in \eqref{eq:newTesting}.
Then
\begin{equation*}
  \mathfrak{T}_r\eqsim \Norm{W_{T,\sigma}^q[\omega]}{L^{r/q}(\sigma)}^{1/q},\qquad
  W_{T,\sigma}^q[\omega]:=\sup_{Q\in\mathcal{D}}   \frac{1_Q}{\sigma(Q)}  \Norm{T_Q(\sigma)}{L^q(\omega)}^q.
\end{equation*}
\end{proposition}

\begin{example}
Let $\omega=\sigma$, and consider the pointwise multiplication operator $T:f\mapsto mf$. Using either localization \eqref{it:globTest} or \eqref{it:locTest} from Proposition~\ref{prop:abstractWolff}, we have $T_Q(\sigma)=1_Q m$ and
\begin{equation*}
  W_{T,\sigma}^q[\sigma]=\sup_{Q\in\mathcal{D}}\frac{1_Q}{\sigma(Q)}\int_Q m^q\ud\sigma=M_\sigma (m^q),
\end{equation*}
where $M_\sigma$ is the dyadic maximal operator relative to the $\sigma$ measure. Thus
\begin{equation*}
  \Norm{W_{T,\sigma}^q[\omega]}{L^{r/q}(\sigma)}^{1/q}
  =\Norm{M_\sigma(m^q)}{L^{r/q}(\sigma)}^{1/q}\eqsim\Norm{m}{L^r(\sigma)}=\Norm{T}{L^p(\sigma)\to L^q(\sigma)}
\end{equation*}
by the boundedness of the maximal operator for $r/q>1$, and H\"older's inequality in the last step. The theory of Wolff potentials is certainly overshooting to describe the boundedness of these simple multiplication operators, but it is interesting to note that such a degenerate case is also naturally covered by our general theory.
\end{example}

Even if $T(\,\cdot\,\sigma)$ is given by \eqref{eq:TandTQ}, its abstract Wolff potential $W_{T,\sigma}^q[\omega]$ may not be pointwise equal or comparable to the discrete Wolff potential $W_{\lambda,\omega}^q[\sigma]$ defined before; however, their relevant norms are equivalent:

\begin{proposition}\label{prop:absVsConcrete}
If $T(\,\cdot\,\sigma)$ is as in \eqref{eq:TandTQ}, then
\begin{equation*}
  \Norm{W_{T,\sigma}^q[\omega]}{L^{r/q}(\sigma)}\eqsim\Norm{W_{\lambda,\omega}^q[\sigma]}{L^{r/q}(\sigma)}.
\end{equation*}
\end{proposition}

This discussion provides the following analysis of the equivalence ``boundedness $\Leftrightarrow$ discrete Wolff bound'' established by Tanaka's Theorem~\ref{thm:Tanaka}: The next few implications are valid for any positive linear operator:
\begin{equation*}
  \text{boundedness }\Rightarrow\text{ sequential testing }\Leftrightarrow
  \text{ abstract Wolff, }\phantom{\Leftrightarrow\text{ discrete Wolff}.}
\end{equation*}
whereas the concrete structure \eqref{eq:TandTQ} is only needed to prove that:
\begin{equation*}
  \text{boundedness } \Leftarrow \text{ sequential testing, }\phantom{\Leftrightarrow}
  \text{ abstract Wolff } \Leftrightarrow\text{ discrete Wolff}.
\end{equation*}

We feel that this unified and general point of view will be helpful in identifying potential-type characterizations for other operators as well.

In the second part of this paper, we apply these methods to study the bilinear version of \eqref{eq:TandTQ},
\begin{equation}\label{eq:TandTQbil}
\begin{split}
 T(f_1\sigma_1,f_2\sigma_2)&=\sum_{Q\in\mathcal{D}}\lambda_Q\int_Q f_1\ud\sigma_1\cdot\int_Q f_2\ud\sigma_2\cdot 1_Q, \quad\lambda_Q \geq0,\\
T_Q(f_1\sigma_1,f_2\sigma_2)&=\sum_{\substack{Q'\in\mathcal{D}:\\Q'\subseteq Q}}\lambda_{Q'}\int_{Q'} f_1\ud\sigma_1\cdot\int_{Q'} f_2\ud\sigma_2\cdot 1_{Q'}.
\end{split}
\end{equation}
and we characterize the boundedness of
\begin{equation}\label{eq:bilBd}
 T(\,\cdot\,\sigma_1,\,\cdot\,\sigma_2):L^{p_1}(\sigma_1)\times L^{p_2}(\sigma_2)\to L^q(\omega)
\end{equation}
for all values of $p_1,p_2,q\in(1,\infty)$. Writing $p_3:=q':=q/(q-1)$, this was previously known in the case that $1/p_i+1/p_j\geq 1$ for all $i\neq j$ by Li and Sun \cite{li2013}, and extended to $1/p_1+1/p_2+1/p_3\geq 1$ by Tanaka \cite{tanaka2014b}, but the case of $1/p_1+1/p_2+1/p_3< 1$ remained open until now.
In this paper we prove the following sequential characterization, and use it to extend the discrete-Wolff-potential-type characterization of the boundedness of \eqref{eq:bilBd} to the previously unknown case of $\frac 1{p_1}+\frac 1{p_2}+\frac 1{p_3}<1$.

\begin{theorem}\label{thm:bilinear}Let $T$ be defined as in \eqref{eq:TandTQbil}. For $p_1,p_2,p_3\in(1,\infty)$ and three measures $\sigma_1,\sigma_2,\sigma_3$, we have
\begin{equation*}
  \| T(\,\cdot\,\sigma_1,\,\cdot\,\sigma_2)\|_{L^{p_1}(\sigma_1)\times L^{p_2}(\sigma_2)\to L^{p_3'}(\sigma_3)}\eqsim \sum_{(i,j,k)\in S_3} \mathfrak{T}_{i,j,k}
  \eqsim \sum_{(i,j,k)\in S_3} \tilde{\mathfrak{T}}_{i,j,k},
\end{equation*}
where $S_3$ is the set of all permutations $(i,j,k)$ of $(1,2,3)$, and the quantities $ \mathfrak{T}_{i,j,k}$ and $ \tilde{\mathfrak{T}}_{i,j,k}$ are defined by
\begin{equation*}
\begin{split}
  \mathfrak{T}_{i,j,k} :=\sup_{\cf_j,\cf_k} &\left\|\bigg\{ \bigg\|\bigg\{ \frac{\|T_{F_j}(\sigma_j,\sigma_k)\|_{L^{p_i'}(\sigma_i)}}{\sigma_j(F_j)^{1/p_j}\sigma_k(F_k)^{1/p_k}}
      \bigg\}_{F_j\in\cf_j : F_j\subseteq F_k}\bigg\|_{\ell^{r_k}} \bigg\}_{F_k\in\cf_k}\right\|_{\ell^r}, \\
  \tilde{\mathfrak{T}}_{i,j,k} :=\sup_{\cf_k,\cf_j(F_k)}
  & \left\|\bigg\{ \bigg\|\bigg\{ \frac{\|T_{F_j}(\sigma_j,\sigma_k)\|_{L^{p_i'}(\sigma_i)}}{\sigma_j(F_j)^{1/p_j}\sigma_k(F_k)^{1/p_k}}
      \bigg\}_{F_j\in\cf_j(F_k)}\bigg\|_{\ell^{r_k}} \bigg\}_{F_k\in\cf_k}\right\|_{\ell^r},
\end{split}
\end{equation*}
where
\begin{itemize}
  \item the supremum $\sup_{\cf_j,\cf_k}$ is over all $\sigma_j$-sparse collections $\cf_j$ and all $\sigma_k$-sparse collections $\cf_k$;
  \item the supremum $\sup_{\cf_k,\cf_j(F_k)}$ is over all $\sigma_k$-sparse collections $\cf_k$ and, for each $F_k\in\cf_k$, over all $\sigma_j$-sparse collections $\cf_j(F_k)$ of cubes contained in $F_k$;
  \item the exponents are defined by
\begin{equation}\label{def:sequentialexponents}
   \frac{1}{r_k}:=\big(1-\frac{1}{p_i}-\frac{1}{p_j}\big)_+, \quad\text{ and } \quad\frac{1}{r }:=\big(1-\frac{1}{p_1}-\frac{1}{p_2} -\frac{1}{p_3}\big)_+
\end{equation}
for all $(i,j,k)\in S_3$.
\end{itemize}
\end{theorem}

Note that trivially $\mathfrak{T}_{i,j,k}\leq\tilde{\mathfrak{T}}_{i,j.k}$, since $\{F_j\in\cf_j:F_j\subseteq F_k\}$ is a special case of a collection $\cf_j(F_k)$. We shall show that even the stronger form of sequential testing (with $\tilde{\mathfrak{T}}_{i,j,k}$) is necessary, while already the weaker form (with $\mathfrak{T}_{i,j,k}$) is sufficient, for the  estimate on $\Norm{T(\,\cdot\,\sigma_1,\,\cdot\,\sigma_2)}{}$.  In fact, analogously to the linear case, we prove that the sequential testing is necessary for the boundedness \eqref{eq:bilBd} for any positive bilinear operator, and that the sequential testing is equivalent to a norm bound for an abstract Wolff potential.

In the previously unknown case $\frac{1}{p_1}+\frac{1}{p_2}+\frac{1}{p_3}<1$, we can reformulate the characterization in terms of a new two-measure Wolff potential as follows:

\begin{theorem}\label{thm:discreteWolff}
Let $\sigma_1,\sigma_2,\sigma_3$ be a locally finite Borel measures on $\mathbb R^n$, and $T(\,\cdot\,\sigma_1, \,\cdot\,\sigma_2)$ be as in \eqref{eq:TandTQbil}.
Then
\begin{equation*}
  \| T(\,\cdot\,\sigma_1, \,\cdot\,\sigma_2)\|_{L^{p_1}(\sigma_1)\times L^{p_2}(\sigma_2)\rightarrow L^{p_3'}(\sigma_3)}
  \eqsim \sum_{(i,j,k)\in S_3}\| \mathcal W_{\sigma_i, \sigma_j}[\sigma_k]^{1/{r_k}}\|_{L^r(\sigma_k)},
\end{equation*}
where $r_k$ and $r$ are defined as in \eqref{def:sequentialexponents}, and the two-measure Wolff potential is defined by
\begin{equation}\label{eq:bilWolff}
  \mathcal W_{\sigma_i, \sigma_j}[\sigma_k]:=  \sum_{Q\in\mathcal D} \lambda_Q \lambda_{Q,\sigma_i}^{p_i'-1}\lambda_{Q,\sigma_j,\sigma_i}^{\frac{r_k}{p_i'}-1}\sigma_i(Q)
  \sigma_j(Q)  1_Q,
\end{equation}
in terms of the coefficients
\begin{equation}\label{eq:defLambdaQsigma}
\begin{split}
\lambda_{Q,\sigma_i} &:= \frac 1{\sigma_i(Q)}\sum_{Q'\subseteq Q}\lambda_{Q'}\sigma_1(Q')\sigma_2(Q')\sigma_3(Q') \\
\lambda_{Q,\sigma_j,\sigma_i}
&:= \frac 1{\sigma_j(Q)}\sum_{Q'\subseteq Q}\lambda_{Q'}\lambda_{Q',\sigma_i}^{p_i'-1}\sigma_1(Q')\sigma_2(Q')\sigma_3(Q'),\,\, j\neq i.
\end{split}
\end{equation}
\end{theorem}

We note that the need of some kind of a two-measure Wolff potential, to describe the boundedness in this case, was already suggested by Tanaka \cite[Remark 1.4]{tanaka2014b} but, as far as we know, this is the first time that an appropriate potential is actually written down explicitly.

While we mainly work with the dyadic model operators, as an application, we give a characterization for the bilinear fractional integral operator $I_\alpha (\,\cdot\,\sigma_1,\,\cdot\,\sigma_2):L^{p_1}(\sigma_1)\times L^{p_2}(\sigma_2)\to L^q(\omega)$ in terms of sequential testing, and of an abstract Wolff potential, which generalizes the results in \cite{li2013}. We also characterize the boundedness of (linearized) maximal operators $M(\cdot \,\sigma\,):L^{p}(\sigma)\to L^q(\omega)$ and $M(\cdot \,\sigma_1,\,\cdot\,\sigma_2):L^{p_1}(\sigma_1)\times L^{p_2}(\sigma_2)\to L^q(\omega)$  in the cases $\frac{1}{p}<\frac{1}{q}$ and $\frac{1}{p_1}+\frac{1}{p_2}<\frac{1}{q}$, respectively. This extends the results in \cite{sawyer1982,sawyer1988,sawyer1996,liSun2013}.

\section{Basic lemmas of positive dyadic analysis}\label{sec:basiclemmas}

For a family $\cf\subseteq\mathcal{D}$, we denote
\begin{equation*}
  \operatorname{ch}_{\cf}(F):=\{\text{maximal }F'\subseteq F: F'\in\cf\},\quad
  E_{\cf}(F):=F\setminus\bigcup_{F'\in\operatorname{ch}_{\cf}(F)}F'.
\end{equation*}
The sets $E_{\cf}(F)$ are pairwise disjoint. Recalling Definition~\ref{def:sparse}, we note in particular that $\cf$ is $\sigma$-sparse if (but not only if) $\sigma(E_{\cf}(F))\geq\frac12\sigma(F)$, or equivalently
\begin{equation*}
  \sum_{F'\in\operatorname{ch}_{\cf}}\sigma(F')\leq\frac12\sigma(F).
\end{equation*}
Although sparseness is slightly more general, this is actually the only case that we need in this paper.

For $Q\in\mathcal{D}$, we denote by
\begin{equation*}
  \pi_{\cf}Q:=\pi_{\cf}^0 Q:=\min\{F\supseteq Q:F\in\cf\},\quad
  \pi^{k+1}_{\cf} Q:=\min\{F\supsetneq \pi_{\cf}^k Q :F\in\cf\}.
\end{equation*}

We sometimes specify a ``stopping family starting at $Q_0$ and defined by $\operatorname{ch}_{\cf}$'', where we give some explicit rule of constructing a disjoint collection $\operatorname{ch}_{\cf}(F)$ of dyadic (strict) subcubes $F'$ of a given $F$. By this we mean that $\cf$ is given as follows: We set $\cf_0:=\{Q_0\}$. Assuming that $\cf_k$ is already chosen, we define
\begin{equation*}
   \cf_{k+1}:=\bigcup_{F\in\cf_k}\operatorname{ch}_{\cf}(F),
\end{equation*}
where the last expression is meaningful by the rule that we specified. Then finally
\begin{equation*}
   \cf:=\bigcup_{k=0}^\infty\cf_k.
\end{equation*}

The dyadic maximal operator $M_{\cd}^\sigma $ adapted to a dyadic grid $\cd$ and a locally finite Borel measure $\sigma$ is defined by
\[
M_{\cd}^\sigma (f)(x):=\sup_{Q\in\cd}\frac{1_Q(x)}{\sigma(Q)}\int_Q f(y)\dsigma(y).
\] We state the boundedness of the maximal operator and (a special case of) the dyadic Carleson embedding theorem as the following well-known lemmas:
\begin{lemma}Let $1<p<\infty$. Let $\sigma$ be a locally finite Borel measure. Then
\[
\|M_{\cd}^\sigma (f)\|_{L^p(\sigma)}\le p' \|f\|_{L^p(\sigma)}.
\]
\end{lemma}
\begin{lemma}\label{lm:carleson}Let
$1 \le p < \infty$. Let $\sigma$ be a locally finite Borel measure.
Suppose that $\cf$ is a $\sigma$-sparse collection.
Then
\[
\Big(\sum_{F\in\cf}(\langle f\rangle_F^\sigma)^p \sigma(F)\Big)^{1/p}\le 2p'\|f\|_{L^{p}(\sigma)}.
\]
\end{lemma}

The next lemma can be thought of as an $L^p$-variant of Pythagoras' theorem for functions adapted to a sparse collection:

\begin{lemma}[\cite{hanninenhytonen2014}]\label{lem:pythagoras}
Let $1 \le p < \infty$. Let $\sigma$ be a locally finite Borel measure. Let $\mathcal S$ be a $\sigma$-sparse collection of dyadic cubes. For each $S \in \mathcal S$, assume that $a_S$ is a non-negative function that is supported on $S$ and
constant on each $S'\in \ch_{\mathcal S}(S)$. Then
\[
  \left(\sum_{S}\|a_S\|_{L^p(\sigma)}^p\right)^{1/p}\le\BNorm{\sum_S a_S}{L^p(\sigma)}\le 3p \left(\sum_{S}\|a_S\|_{L^p(\sigma)}^p\right)^{1/p}.
\]
\end{lemma}

\begin{proof}
We note that the first estimate is simply $\Norm{\ }{\ell^p}\leq\Norm{\ }{\ell^1}$; the second one is found in \cite{hanninenhytonen2014}.
\end{proof}


\begin{lemma}\label{lem:dyadicSup}Let $1<s<\infty$. Let $\sigma$ be a locally finite Borel measure.
Let
\begin{equation}\label{eq:dyadicSup}
  \psi=\sup_{Q\in\mathcal{D}}\frac{\tau_Q}{\sigma(Q)} 1_Q,
\end{equation}
where the coefficients are superadditive in the sense that
\begin{equation}\label{eq:tauSuperadditive}
  \sum_{i=1}^\infty\tau_{Q_i}\leq \tau_Q
\end{equation}
whenever $Q_i\subseteq Q$ are disjoint. Then
\begin{equation}\label{eq:dyadicSupEquiv}
  \Norm{\psi}{L^s(\sigma)}
  \eqsim\sup_{\cf}\BNorm{\sup_{F\in\cf}\frac{\tau_F}{\sigma(F)}1_F}{L^s(\sigma)}
  \eqsim\sup_{\cf}\Big(\sum_{F\in\cf}\Big[\frac{\tau_F}{\sigma(F)}\Big]^s\sigma(F)\Big)^{1/s},
\end{equation}
where the supremums are over all $\sigma$-sparse subcollections $\cf\subseteq\mathcal{D}$.
\end{lemma}

\begin{proof}
Choosing an increasing sequence $Q_N$ that exhausts the space, it is clear that $\psi_{Q_N}:=\sup_{Q\subseteq Q_N}\tau_Q/\sigma(Q)\cdot 1_Q$ increases pointwise to $\psi$, and therefore $\Norm{\psi_{Q_N}}{L^s(\sigma)}\uparrow\Norm{\psi}{L^s(\sigma)}$ by monotone convergence. So it is enough to dominate $\psi_{Q_N}$ in place of $\psi$ in the first ``$\lesssim$''.

Let $\cf$ be the stopping cubes starting at $Q_N$ and defined by
\begin{equation*}
  \operatorname{ch}_{\cf}(F):=\Big\{\text{maximal }F'\subsetneq F:\frac{\tau_{F'}}{\sigma(F')}>2\frac{\tau_F}{\sigma(F)}\Big\}.
\end{equation*}
Then sparseness follows from the assumed superadditivity:
\begin{equation*}
  \sum_{F'\in\operatorname{ch}_{\cf}(F)}\sigma(F')
  \leq\sum_{F'\in\operatorname{ch}_{\cf}(F)}\frac{\tau_{F'}\sigma(F)}{2\tau_F}
  \leq\frac{1}{2}\sigma(F),
\end{equation*}
and $\psi_{Q_N}\leq 2\sup_{F\in\cf}\tau_F/\sigma(F)\cdot 1_F$ by the stopping condition. This shows the first ``$\lesssim$'' in \eqref{eq:dyadicSupEquiv}, and the second is immediate from $\Norm{\
}{\ell^\infty}\leq\Norm{\ }{\ell^s}$.

To see that the right-most term in \eqref{eq:dyadicSupEquiv} is bounded by the left-most one, we observe that $\tau_F/\sigma(F)\leq\inf_F\psi$ and $\sigma(F)\leq 2\sigma(E_{\cf}(F))$ by sparseness, which shows that
\begin{equation*}
  \sum_{F\in\cf}\Big[\frac{\tau_F}{\sigma(F)}\Big]^s\sigma(F)
  \leq 2\sum_{F\in\cf}\int_{E_{\cf}(F)}\psi^s\ud\sigma\leq 2\Norm{\psi}{L^s(\sigma)}^s,
\end{equation*}
and this completes the proof.
\end{proof}

We also need the following lemma.

\begin{lemma}[\cite{cov2004}, Proposition~2.2]\label{lem:dyadicSum}Let $1<s<\infty$. Let $\sigma$ be a locally finite Borel measure.
Let
\begin{equation}\label{eq:dyadicSum}
  \phi=\sum_{Q\in\mathcal{D}}\alpha_Q 1_Q,\qquad
  \phi_Q:=\sum_{\substack{Q'\subseteq Q}}\alpha_{Q'} 1_{Q'}.
\end{equation}
Then
\begin{equation*}
  \Norm{\phi}{L^s(\sigma)}
  \eqsim\Big(\sum_{Q\in\mathcal{D}}\alpha_Q(\ave{\phi_Q}_Q^\sigma)^{s-1}\sigma(Q)\Big)^{1/s}
  \eqsim\BNorm{\sup_{Q\in\mathcal{D}}1_Q\ave{\phi_Q}_Q^\sigma}{L^s(\sigma)}.
\end{equation*}
\end{lemma}

\section{Two-weight $T1$ theorem for positive operators: the linear case}\label{sec:T1}

In this section we prove Theorem~\ref{thm:seqTesting} and Proposition~\ref{prop:abstractWolff} and \ref{prop:absVsConcrete} discussed in the Introduction.

\subsection{Necessity of sequential testing}

We begin with a more general version of ``sequential boundedness'' of positive linear operators between $L^p$ spaces.

\begin{proposition}\label{prop:nec}
Let $p,q\in(1,\infty)$, and $T(\,\cdot\,\sigma):L^p(\sigma)\to L^q(\omega)$ be a positive bounded linear operator. Then for $r$ defined in \eqref{eq:rDef} and all nonnegative functions $\phi_F$, we have
\begin{equation*}
\begin{split}
  &\BNorm{\Big\{\Norm{T(\phi_F\sigma)}{L^q(\omega)}\Big\}_{F\in\cf}}{\ell^r}\\
  &\leq \Norm{T(\,\cdot\,\sigma)}{L^p(\sigma)\to L^q(\omega)}\sup\Big\{\BNorm{\sum_F \beta_F\phi_F}{L^p(\sigma)}:\sum_F \beta_F^p\leq 1\Big\}.
\end{split}
\end{equation*}
\end{proposition}

\begin{proof}
We first note that this is obvious for $p\leq q$, when $r=\infty$, so we concentrate on $p>q$.
We rewrite the left side by duality and observe that $(r/q)'=p/q$:
\begin{equation*}
\begin{split}
  \Big[\sum_F & \Norm{T(\phi_F\sigma)}{L^q(\omega)}^r\Big]^{1/r}
    =\Big[\sum_F \Big(\Norm{T(\phi_F\sigma)}{L^q(\omega)}^q\Big)^{r/q}\Big]^{1/r} \\
  & =\sup\Big\{\Big[\sum_F \alpha_F \Norm{T(\phi_F\sigma)}{L^q(\omega)}^q\Big]^{1/q} : \sum_F\alpha_F^{(r/q)'}\leq 1\Big\} \\
  & =\sup\Big\{\Big[\sum_F  \Norm{T(\alpha_F^{1/q}\phi_F\sigma)}{L^q(\omega)}^q\Big]^{1/q} : \sum_F\alpha_F^{p/q}\leq 1\Big\},
\end{split}
\end{equation*}
Denoting $\beta_F:=\alpha_F^{1/q}$, where these numbers satisfy $\sum_F\beta_F^p\leq 1$, we continue with
\begin{equation*}
\begin{split}
  \sum_F &\Norm{T(\beta_F\phi_F\sigma)}{L^q(\omega)}^q
  =\int \sum_F\ [T(\beta_F \phi_F\sigma)]^q\ud\omega \\
  &\leq\int\Big[\sum_F T(\beta_F\phi_F\sigma)\Big]^q\ud\omega
  =\BNorm{T\Big(\sum_F\beta_F\phi_F\sigma\Big)}{L^q(\omega)}^q \\
  &\leq\Norm{T(\,\cdot\,\sigma)}{L^p(\sigma)\to L^q(\omega)}^q\BNorm{\sum_F\beta_F\phi_F}{L^p(\sigma)}^q.\qedhere
\end{split}
\end{equation*}
\end{proof}

The necessity of sequential testing in Theorem~\ref{thm:seqTesting} is a consequence of the following corollary and the trivial pointwise bound $T_F(\sigma)\leq T(1_F\sigma)$:

\begin{corollary}\label{cor:necessityseq}
Let $p,q\in(1,\infty)$ and $T(\,\cdot\,\sigma):L^p(\sigma)\to L^q(\omega)$ be a positive bounded linear operator. Then for $r$ defined in \eqref{eq:rDef} and any sparse collection $\cf$ with respect to $\sigma$, we have
\begin{equation*}
  \BNorm{\Big\{\frac{1}{\sigma(F)^{1/p}}\Norm{T(1_F\sigma)}{L^q(\omega)}\Big\}_{F\in\cf}}{\ell^r}
  \leq 3p\cdot \Norm{T(\,\cdot\,\sigma)}{L^p(\sigma)\to L^q(\omega)}.
\end{equation*}
\end{corollary}

\begin{proof}
Applying to Proposition~\ref{prop:nec} with $\phi_F=1_F/\sigma(F)^{1/p}$, we are reduced to estimating
\begin{equation*}
\begin{split}
  &\BNorm{\sum_F\frac{\beta_F}{\sigma(F)^{1/p}}1_F}{L^p(\sigma)}\qquad\text{where }\sum_F\beta_F^p\leq 1 \\
  &\leq 3p\Big(\sum_F\frac{\beta_F^p}{\sigma(F)}\sigma(F)\Big)^{1/p}\qquad\text{by Lemma \ref{lem:pythagoras}} \\
  &=3p\Big(\sum_F\beta_F^p\Big)^{1/p}\leq 3p.\qedhere
\end{split}
\end{equation*}
\end{proof}

\subsection{Sufficiency of sequential testing}

We turn to the other direction of Theorem~\ref{thm:seqTesting} by showing the boundedness of $T(\,\cdot\,\sigma)$ under the sequential testing assumptions. The argument is rather short once we take for granted the following lemma, which summarizes the part of the argument common to both previous Theorems~\ref{thm:LSU} (as proven in \cite{hytonen2012a}) and \ref{thm:Tanaka}:

\begin{lemma}\label{lem:linearCore}
Let $T$ and $T_Q$ be as in \eqref{eq:TandTQ}.
For any $Q_0\in\mathcal{D}$, there are subsets $\cf,\mathcal{G}\subseteq\mathcal{D}$, respectively sparse with respect to $\sigma,\omega$, such that
\begin{equation}\label{eq:Tdecompo}
\begin{split}
  \pair{T_{Q_0}(f\sigma)}{g\omega}
  &\leq 2\sum_{F\in\cf}\ave{f}_F^\sigma\Norm{T_F(\sigma)}{L^q(\omega)}\Norm{g_F}{L^{q'}(\omega)} \\
   &\qquad +2\sum_{G\in\mathcal{G}}\ave{g}_G^\omega\Norm{f_F}{L^p(\sigma)}\Norm{T_F(\omega)}{L^{q'}(\omega)},
\end{split}
\end{equation}
where
\begin{equation*}
\begin{split}
  \Big(\sum_{F\in\cf}(\ave{f}_F^\sigma)^p\sigma(F)\Big)^{1/p}\lesssim\Norm{f}{L^p(\sigma)},\qquad
  \Big(\sum_{F\in\cf}\Norm{g_F}{L^{q'}(\omega)}^{q'}\Big)^{1/q'}\lesssim\Norm{g}{L^{q'}(\omega)}, \\
  \Big(\sum_{G\in\mathcal{G}}(\ave{g}_G^\omega)^{q'}\omega(G)\Big)^{1/q'}\lesssim\Norm{g}{L^{q'}(\omega)},\qquad
  \Big(\sum_{G\in\mathcal{G}}\Norm{f_G}{L^{p}(\sigma)}^{p}\Big)^{1/p}\lesssim\Norm{f}{L^{p}(\sigma)}.
\end{split}
\end{equation*}
\end{lemma}

\begin{proof}
This is implicitly contained as an intermediate step in the proof of \cite[Theorem 6.1]{hytonen2012a}.
\end{proof}

We observe that the number $r$ is chosen in such a way that
\begin{equation*}
  \frac{1}{p}+\frac{1}{r}+\frac{1}{q'}=\frac{1}{p}+\Big(\frac{1}{q}-\frac{1}{p}\Big)_+ +\frac{1}{q'}\geq 1
\end{equation*}
in each case. Thus, using H\"older's inequality (and the monotonicity of the $\ell^s$ norms in the case of strict inequality above), we have
\begin{equation*}
\begin{split}
  \sum_{F\in\cf} &\ave{f}_F^\sigma\Norm{T_F(\sigma)}{L^q(\omega)}\Norm{g_F}{L^{q'}(\omega)} \\
  &=\sum_{F\in\cf}\ave{f}_F^\sigma\sigma(F)^{1/p}\frac{\Norm{T_F(\sigma)}{L^q(\omega)}}{\sigma(F)^{1/p}}\Norm{g_F}{L^{q'}(\omega)} \\
  &\leq \Big(\sum_{F\in\cf}(\ave{f}_F^\sigma)^p\sigma(F)\Big)^{1/p}
  \BNorm{ \Big\{ \frac{\Norm{T_F(\sigma)}{L^q(\omega)}}{\sigma(F)^{1/p}} \Big\}_{F\in\cf} }{\ell^r}
  \Big(\sum_{F\in\cf}\Norm{g_F}{L^{q'}(\omega)}^{q'}\Big)^{1/q'} \\
  &\lesssim\Norm{f}{L^p(\sigma)}\cdot\mathfrak{T}_r\cdot\Norm{g}{L^{q'}(\omega)}.
\end{split}
\end{equation*}
A symmetric argument for the second term in \eqref{eq:Tdecompo} shows that
\begin{equation*}
  \pair{T_{Q_0}(f\sigma)}{g\omega}\lesssim(\mathfrak{T}_r+\mathfrak{T}_r^*)\Norm{f}{L^p(\sigma)}\Norm{g}{L^{q'}(\omega)},
\end{equation*}
and taking $Q_0$ as large as we like it follows by monotone convergence and duality that
\begin{equation*}
   \Norm{T(f\sigma)}{L^q(\omega)}\lesssim (\mathfrak{T}_r+\mathfrak{T}_r^*)\Norm{f}{L^p(\sigma)}.
\end{equation*}
This completes the proof of Theorem~\ref{thm:seqTesting}.

\subsection{Abstract Wolff potential}\label{subsec:abstractwolffpotential}

We prove Propositions~\ref{prop:abstractWolff} and \ref{prop:absVsConcrete} concerning the abstract Wolff potential $W_{T,\sigma}^q$ and its relation to the discrete Wolff potential~$W_{\lambda,\sigma}^q$.

\begin{proof}[Proof of Proposition~\ref{prop:abstractWolff}]
Recall that
\begin{equation*}
  \mathfrak{T}_r
  :=\sup_{\cf}\BNorm{\Big\{\frac{\Norm{T_F(\sigma)}{L^q(\omega)}}{\sigma(F)^{1/p}}\}_{F\in\cf}}{\ell^r}
  =\sup_{\cf}\Big(\sum_{F\in\cf}\Big[\frac{\Norm{T_F(\sigma)}{L^q(\omega)}^q}{\sigma(F)}\Big]^{r/q}\sigma(F)\Big)^{1/r}.
\end{equation*}
We denote $\tau_Q:=\Norm{T_Q(\sigma)}{L^q(\omega)}^q$. Once we check that these satisfy the superadditivity \eqref{eq:tauSuperadditive}, Proposition~\ref{prop:abstractWolff} is a direct consequence of Lemma~\ref{lem:dyadicSup}.

In the first case of $T_Q=T(1_Q\cdot)$, we have
\begin{equation*}
\begin{split}
  \sum_i\Norm{T(1_{Q_i}\sigma)}{L^q(\omega)}^q
  &=\int \sum_i(T(1_{Q_i}\sigma))^q\ud\omega
  \leq \int\Big[ \sum_i T(1_{Q_i}\sigma)\Big]^q\ud\omega \\
  &=\int\Big[T\Big(\sum_i 1_{Q_i}\sigma\Big)\Big]^q\ud\omega
  \leq\int[T(1_{Q}\sigma)]^q\ud\omega=\Norm{T(1_Q\sigma)}{L^q(\omega)}^q.
\end{split}
\end{equation*}
The other two cases follow by first estimating $T_{Q_i}(\sigma)\leq T_{Q}(1_{Q_i}\sigma)$ and then applying the first case to the operator $T_Q$ in place of $T$.
\end{proof}

\begin{proof}[Proof of Proposition~\ref{prop:absVsConcrete}]
We need to show that the $L^{r/q}(\sigma)$ norms of
\begin{equation*}
  W_{T,\sigma}^q[\omega]:=\sup_{Q\in\mathcal{D}}\frac{1_Q}{\sigma(Q)}\Norm{T_Q(\sigma)}{L^q(\omega)}^q
  =:\sup_{Q\in\mathcal{D}}\frac{1_Q}{\sigma(Q)}\tau_Q
\end{equation*}
and
\begin{equation*}
  W_{\lambda,\omega}^q[\sigma]
  :=\sum_{Q\in\mathcal{D}}\lambda_Q\omega(Q)1_Q\Big(\frac{1}{\omega(Q)}\sum_{Q'\subseteq Q}\lambda_{Q'}\sigma(Q')\omega(Q')\Big)^{q-1}
    =:\sum_{Q\in\mathcal{D}}\alpha_Q 1_Q
\end{equation*}
are comparable.

From Lemma~\ref{lem:dyadicSum} we deduce that
\begin{equation}\label{eq:WolffObs}
\begin{split}
  \tau_R
  &=\int_R\Big(\sum_{Q\subseteq R}\lambda_Q\sigma(Q)1_Q\Big)^q\ud\omega \\
  &\eqsim \sum_{Q\subseteq R}\lambda_Q\sigma(Q)\omega(Q)\Big(\frac{1}{\omega(Q)}\sum_{Q'\subseteq Q}\lambda_{Q'}\sigma(Q')\omega(Q')\Big)^{q-1} \\
  &=\int_R \Big(\sum_{Q\subseteq R}\alpha_Q 1_Q\Big)\ud\sigma.
\end{split}
\end{equation}
Hence
\begin{equation*}
\begin{split}
  \Norm{W_{T,\sigma}^q[\omega]}{L^{r/q}(\sigma)}
  &=\BNorm{\sup_{R\in\mathcal{D}}\frac{1_R}{\sigma(R)}\tau_R}{L^{r/q}(\sigma)} \\
  &\eqsim\BNorm{\sup_{R\in\mathcal{D}}\frac{1_R}{\sigma(R)}\int_R\Big(\sum_{Q\subseteq R}\alpha_Q 1_Q\Big)}{L^{r/q}(\sigma)}
    \qquad\text{by \eqref{eq:WolffObs}} \\
  &\eqsim\BNorm{\sum_{Q\in\mathcal{D}}\alpha_Q 1_Q}{L^{r/q}(\sigma)} \qquad\text{by Lemma~\ref{lem:dyadicSum}} \\
  &=\Norm{W_{\lambda,\omega}^q[\sigma]}{L^{r/q}(\sigma)}.\qedhere
\end{split}
\end{equation*}
\end{proof}

\subsection{Insufficiency of Sawyer testing for $q<p$}

Here we present a concrete example of an operator as in \eqref{eq:TandTQ}, which satisfies the testing conditions \eqref{eq:SawyerTesting}, but fails to be bounded from $L^p$ to $L^q$. For this example, it is enough to consider both $\sigma$ and $\omega$ eqaul to the Lebesgue measure on $\R^n$. This justifies the use of the more complicated sequential testing in Theorem~\ref{thm:seqTesting}.

Let $\mathcal{C}$ be an infinite chain of dyadic cubes, where any two satisfy $C'\subsetneq C$. Let
\begin{equation*}
  Tf:=\sum_{Q\in\mathcal{C}}\abs{Q}^{-1/r}\fint_Q f\ud x\cdot 1_Q.
\end{equation*}
Then
\begin{equation*}
\begin{split}
  \Norm{T_F 1}{L^q}=\BNorm{\sum_{\substack{Q\in\mathcal{C}\\ Q\subseteq F}}\abs{Q}^{-1/r}1_Q}{L^q}
  &\eqsim\Big(\sum_{\substack{Q\in\mathcal{C}\\ Q\subseteq F}}\abs{Q}^{-q/r}\abs{Q}\Big)^{1/q} \\
  &=\Big(\sum_{\substack{Q\in\mathcal{C}\\ Q\subseteq F}}\abs{Q}^{q/p}\Big)^{1/q}
  \lesssim(\abs{F}^{q/p})^{1/q}=\abs{F}^{1/p},
\end{split}
\end{equation*}
where $\eqsim$ follows from the sparseness of $\mathcal{C}$ via Lemma~\ref{lem:pythagoras}, and $\lesssim$ from the fact that the numbers $\abs{Q}^{q/p}$, for $Q$ in the chain, form a geometric progression with maximal element at most $\abs{F}^{q/p}$. A similar computation with $(p,q)$ replaced by $(q',p')$ shows that $\Norm{T_F 1}{L^{p'}}\lesssim\abs{F}^{1/q'}$. Thus this operator satisfies the Sawyer testing conditions.

We will now show that $T$ is not bounded from $L^p$ to $L^q$. Consider $f=\sum_{Q\in\mathcal{C}}a_Q 1_Q$ with $a_Q\geq 0$ so that, again by sparseness and Lemma~\ref{lem:pythagoras},
\begin{equation*}
  \Norm{f}{L^p}\eqsim\Big(\sum_{Q\in\mathcal{C}}a_Q^p\abs{Q}\Big)^{1/p}=:\Big(\sum_{Q\in\mathcal{C}}b_Q^p\Big)^{1/p},\qquad b_Q:=a_Q\abs{Q}^{1/p}.
\end{equation*}
Then $\fint_Q f\ud x\geq a_Q$, and thus
\begin{equation*}
\begin{split}
  \Norm{Tf}{L^q}
  \geq\BNorm{\sum_{Q\in\mathcal{C}}\abs{Q}^{-1/r}a_Q 1_Q}{L^q}
  &\eqsim\Big(\sum_{Q\in\mathcal{C}}\abs{Q}^{-q/r}a_Q^q\abs{Q}\Big)^{1/q} \\
  &=\Big(\sum_{Q\in\mathcal{C}}a_Q^q\abs{Q}^{q/p}\Big)^{1/q}=\Big(\sum_{Q\in\mathcal{C}}b_Q^q\Big)^{1/q}.
\end{split}
\end{equation*}
It suffices to pick a sequence $(b_Q)_{Q\in\mathcal{C}}\in\ell^p\setminus\ell^q$ to conclude the example.

\section{Two-weight $T1$ theorem for positive operators: the bilinear case}

We now turn to the bilinear case and begin with some general observations concerning a bilinear operator
\begin{equation*}
  T(\,\cdot\,\sigma_1, \,\cdot\,\sigma_2):L^{p_1}(\sigma_1)\times L^{p_2}(\sigma_2)\to L^{p_3'}(\sigma_3).
\end{equation*}
Its norm is defined as the least constant $\|T(\,\cdot\,\sigma_1,\,\cdot\,\sigma_2)\|_{L^{p_1}(\sigma_1)\times L^{p_2}(\sigma_2)\rightarrow L^{p_3'}(\sigma_3)}$  in the inequality
\begin{equation*}
\begin{split}
  &\|T(f_1\sigma_1,f_2\sigma_2)\|_{L^{p_3'}(\sigma_3)} \\
  &\leq \|T(\,\cdot\,\sigma_1,\,\cdot\,\sigma_2)\|_{L^{p_1}(\sigma_1)\times L^{p_2}(\sigma_2)\rightarrow L^{p_3'}(\sigma_3)}
     \| f_1 \|_{L^{p_1}(\sigma_1)}\|f_2 \|_{L^{p_2}(\sigma_2)},
\end{split}
\end{equation*}
which, by duality, equals to the least constant $\| T\|$ in the inequality
$$
\int T(f_1\sigma_1,f_2\sigma_2)  f_3 \dsigma_3 \leq \| T\|\, \|f_1\|_{L^{p_1}(\sigma_1)} \| f_2 \|_{L^{p_2}(\sigma_2)}\|f_3 \|_{L^{p_3}(\sigma_3)}.
$$
The partial adjoint $T(\,\cdot\, \sigma_2, \,\cdot\, \sigma_3):L^{p_2}(\sigma_2)\times L^{p_3}(\sigma_3)\to L^{p_1'}(\sigma_1)$ is defined as the bilinear operator satisfying
$$
\int T(f_2\sigma_2, f_3\sigma_3) f_1\dsigma_1=\int T(f_1\sigma_1,f_2\sigma_2)  f_3\,\dsigma_3,
$$
and the partial adjoint $T(\,\cdot\, \sigma_1, \,\cdot\, \sigma_3):L^{p_1}(\sigma_1)\times L^{p_3}(\sigma_3)\to L^{p_2'}(\sigma_2)$ similarly. By definition,
\begin{equation*}
\begin{split}
\| T\| &=\|T(\,\cdot\,\sigma_1,\,\cdot\,\sigma_2)\|_{L^{p_1}(\sigma_1)\times L^{p_2}(\sigma_2)\rightarrow L^{p_3'}(\sigma_3)}\\
  &=\|T(\,\cdot\,\sigma_2,\,\cdot\,\sigma_3)\|_{L^{p_2}(\sigma_2)\times L^{p_3}(\sigma_3)\rightarrow L^{p_1'}(\sigma_1)} \\
  &=\|T(\,\cdot\,\sigma_1,\,\cdot\,\sigma_3)\|_{L^{p_1}(\sigma_1)\times L^{p_3}(\sigma_3)\rightarrow L^{p_2'}(\sigma_2)}.
\end{split}
\end{equation*}

\subsection{Necessity of sequential testing}\label{subsec:nec-sequential}
We prove one direction of Theorem~\ref{thm:bilinear}, namely, the necessity of the sequential testing, even in the stronger form involving the quantities $\tilde{\mathfrak{T}}_{i,j,k}$. As in the linear case, this necessity statement is a more general result about positive bilinear operators, and does not assume the specific structure \eqref{eq:TandTQbil}.

Let $T(\,\cdot\,\sigma_2, \,\cdot\,\sigma_3):L^{p_2}(\sigma_2)\times L^{p_3}(\sigma_3)\to L^{p_1'} (\sigma_1)$ be a (general) positive bilinear operator. The localized operator $T_Q(\sigma_2,\sigma_3)$ can be understood as $T(\cdot 1_Q \sigma_2 , \cdot 1_Q  \sigma_3)$  or $ 1_Q T(\cdot 1_Q \sigma_2 , \cdot  1_Q \sigma_3)$.

Let  $\cf_3$ be a $\sigma_3$-sparse collection. Then, for each $F_3\in \cf_3$, let $\cf_2(F_3)$ be a $\sigma_2$-sparse collection whose cubes are contained in $F_3$. In this section we prove that
\begin{equation}\label{eq:strongnec}
\begin{split}
&\left\|\bigg\{ \bigg\|\bigg\{ \frac{\|T_{F_2}(\sigma_2,\sigma_3)\|_{L^{p_1'}(\sigma_1)}}{\sigma_2(F_2)^{1/p_2}\sigma_3(F_3)^{1/p_3}} \bigg\}_{F_2\in\cf_2(F_3)}\bigg\|_{\ell^{r_3}} \bigg\}_{F_3\in\cf_3}\right\|_{\ell^r}\\
&\leq 9 p_2p_3\|T(\,\cdot\,\sigma_2,\,\cdot\,\sigma_3)\|_{L^{p_2}(\sigma_2)\times L^{p_3}(\sigma_3)\rightarrow L^{p_1'}(\sigma_1)}
 =9p_2 p_3\|T\|.
\end{split}
\end{equation}

Let us notice that if $r=\infty$, this follows from the linear case (Corollary~\ref{cor:necessityseq}) applied to the operator
\begin{equation*}
  \frac{T(\,\cdot\,\sigma_2,1_{F_3}\sigma_3)}{\sigma(F_3)^{1/p_3}}: L^{p_2}(\sigma_2)\to L^{p_1'}(\sigma_1),
\end{equation*}
so we may concentrate on $r<\infty$, in which case also $r_3\leq r<\infty$. It is useful to observe the algebraic relations
\begin{equation}\label{eq:alg}
  \Big(\frac{r_3}{p_1'}\Big)'=\frac{p_2}{p_1'},\qquad
  \Big(\frac{r}{p_1'}\Big)'=\frac{p_2p_3}{p_1'(p_2+p_3)}.
\end{equation}
in this case.

By duality with respect to both $\ell^{r_3/p_1'}$ and $\ell^{r/p_1'}$ norms, it suffices to prove that
\begin{equation}\label{eq:bilNecToProve}
   \sum_{F_3\in\cf_3}
   \beta_{F_3}\sum_{F_2\in\cf_2(F_3)}\alpha_{F_2,F_3}
   \frac{\|T_{F_2}(\sigma_2,\sigma_3)\|_{L^{p_1'}(\sigma_1)}^{p_1'}  }{\sigma_2(F_2)^{p_1'/p_2}\sigma_3(F_3)^{p_1'/p_3}}
   \leq (9 p_2p_3\|T\|)^{p_1'}
\end{equation}
for all sequences
\begin{equation*}
  \sum_{F_3\in\cf_3}\beta_{F_3}^{(r/p_1')'}\leq 1,\qquad
  \sum_{F_2\in\cf_2(F_3)}\alpha_{F_2,F_3}^{(r_3/p_1')'}\leq 1.
\end{equation*}

In fact, we have
\begin{equation*}
\begin{split}
   &LHS\eqref{eq:bilNecToProve} \\
  &\leq\BNorm{\sum_{F_3\in\cf_3}\beta_{F_3}^{1/p_1'}\sum_{F_2\in\cf_2(F_3)}\alpha_{F_2,F_3}^{1/p_1'}
     \frac{T(1_{F_2}\sigma_2,1_{F_3}\sigma_3)}{\sigma_2(F_2)^{1/p_2}\sigma_3(F_3)^{1/p_3}} }{L^{p_1'}(\sigma_1)}^{p_1'} \\
   &\leq\Big\|T\Big(\sup_{F_3\in \cf_3}\beta_{F_3}^{\frac{p_3}{p_1'(p_2+p_3)}}\sum_{F_2\in\cf_2( F_3)}
       \frac{  \alpha_{F_2, F_3}^{1/{p_1'}} 1_{F_2}}{\sigma_2(F_2)^{\frac 1{p_2}} }\sigma_2,
     \sum_{F_3\in \cf_3} \frac{\beta_{F_3}^{\frac{p_2}{p_1'(p_2+p_3)}} 1_{F_3}}{ \sigma_3(F_3)^{\frac{1}{p_3}} } \sigma_3\Big)
     \Big\|^{p_1'}_{L^{p_1'}(\sigma_1)}\\
   &\leq\|T\|^{p_1'}\BNorm{\sup_{F_3\in \cf_3}\beta_{F_3}^{\frac{p_3}{p_1'(p_2+p_3)}}\sum_{F_2\in\cf_2( F_3)}
       \frac{  \alpha_{F_2, F_3}^{1/{p_1'}} 1_{F_2}}{\sigma_2(F_2)^{\frac 1{p_2}} }}{L^{p_2}(\sigma_2)}^{p_1'} \times \\
&\qquad\qquad \times
    \BNorm{ \sum_{F_3\in \cf_3} \frac{\beta_{F_3}^{\frac{p_2}{p_1'(p_2+p_3)}} 1_{F_3}}{ \sigma_3(F_3)^{\frac{1}{p_3}} } }{L^{p_3}(\sigma_3)}^{p_1'} \\
\end{split}
\end{equation*}

By using $\Norm{\ }{\ell^\infty}\leq\Norm{\ }{\ell^{p_2}}$,  the relations \eqref{eq:alg}, and Lemma \ref{lem:pythagoras}, we get
\begin{eqnarray*}
&&\bigg\|\sup_{F_3\in \cf_3}\beta_{F_3}^{\frac{p_3}{p_1'(p_2+p_3)}}\sum_{F_2\in\cf_2( F_3)}\alpha_{F_2, F_3}^{1/{p_1'}}\sigma_2(F_2)^{-\frac 1{p_2}}  1_{F_2} \bigg\|_{L^{p_2}(\sigma_2)}^{p_1'}\\
&\le& \bigg(\sum_{F_3\in \cf_3}\beta_{F_3}^{(\frac{r }{p_1'})'} \BNorm{\sum_{F_2\in\cf_2( F_3)}\alpha_{F_2, F_3}^{1/{p_1'}}\sigma_2(F_2)^{-\frac 1{p_2}}  1_{F_2}}{L^{p_2}(\sigma_2)}^{p_2}\bigg)^{p_1'/{p_2}}\\
&\le& \bigg(\sum_{F_3\in \cf_3}\beta_{F_3}^{(\frac{r}{p_1'})'}  (3p_2)^{p_2}\sum_{F_2\in\cf_2(F_3)} \alpha_{F_2,F_3}^{(\frac{r_3}{p_1'})'} \bigg)^{p_1'/{p_2}}
\le  (3p_2)^{p_1'}.
\end{eqnarray*}
Similarly, by Lemma \ref{lem:pythagoras},
\[
\bigg\|    \sum_{F_3\in \cf_3}\beta_{F_3}^{\frac{p_2}{p_1'(p_2+p_3)}}\sigma_3(F_3)^{-\frac{1}{p_3}}  1_{F_3}      \bigg\|_{L^{p_3}(\sigma_3)}^{p_1'}\le (3p_3)^{p_1'}.
\]
Combining the arguments together, we get \eqref{eq:bilNecToProve} and thus \eqref{eq:strongnec}.

\subsection{Sufficiency of sequential testing}\label{subsec:suf-sequential}

We turn to the other direction of Theorem~\ref{thm:bilinear}. In this subsection we prove that the weaker version of the testing conditions, defined in terms of the testing constants $\mathfrak{T}_{i,j,k}$, is already sufficient for the boundedness of the bilinear operator $T$. (For readers familiar with the proof of Lemma~\ref{lem:linearCore} from the linear case, which we simply borrowed from earlier papers, it is good to observe that the bilinear argument that follows, although analogous, is set up in a slightly different way, and would reduce to a slight variation of the earlier existing arguments in the linear case.)

The terminology and the notation are fixed in Section \ref{sec:basiclemmas}.
Let $\cd$ be a collection of dyadic cubes such that for some $Q_0\in\cd$ we have $Q\subseteq Q_0$ for all $Q\in\cd$. For each $i=1,2,3$, let $\cf_i$ be the stopping family starting at $Q_0$ and defined by the stopping condition
\begin{equation*}
  \ch_{\cf_i}(F_i):=\{F'_i\in\cd : \text{$F_i'$ maximal such that $\langle f_i \rangle^{\sigma_i}_{F'_i}>2 \langle f_i \rangle^{\sigma_i}_{F_i}$}\}.
\end{equation*}
Each collection $\cf_i$ is $\sigma_i$-sparse, since
\begin{equation*}
  \sum_{F_i'\in \ch_{\cf_i}(F_i)} \sigma_i(F_i')\leq \frac{1}{2} \frac{\sum_{F_i'\in \ch_{\cf_i}(F_i)} \int_{F_i'} f \dsigma}{\int_{F_i} f \dsigma} \sigma_i(F_i)
  \leq \frac{1}{2}  \sigma_i(F_i).
\end{equation*}
The $\cf_i$-stopping parent $\pi_{\cf_i}(Q)$ of a cube $Q$ is defined by
$
\pi_{\cf_i}(Q):=\{F_i \in\cf_i : \text{$F_i$ minimal such that $F_i\supseteq Q$}\}.
$
By the stopping condition, for every cube $Q$ we have $\langle f_i \rangle^{\sigma_i}_{Q}\leq 2 \langle f_i \rangle^{\sigma_i}_{\pi_{\cf_i}(Q)}$.


By rearranging the summation according to the stopping parents, we have
\begin{equation*}
\begin{split}
&\sum_{Q\in\cd} \lambda_Q \int_Q f_1 \dsigma_1 \int_Q f_2 \dsigma_2 \int_Q f_3 \dsigma_3\\
&=\sum_{ (F_1,F_2,F_3)\in(\cf_1, \cf_2, \cf_3)} \sum_{\substack{Q\in\cd:\\ \pi(Q)=(F_1,F_2,F_3)}} \lambda_Q \int_Q f_1 \dsigma_1 \int_Q f_2 \dsigma_2 \int_Q f_3 \dsigma_3\\
&\leq 8 \sum_{(F_1,F_2,F_3)\in(\cf_1, \cf_2, \cf_3)} \langle f_3 \rangle^{\sigma_3}_{F_3} \langle f_2 \rangle^{\sigma_2}_{F_2}  \langle f_1 \rangle^{\sigma_1}_{F_1} \sum_{\substack{Q\in\cd:\\ \pi(Q)=(F_1,F_2,F_3)}} \lambda_Q \int_Q  \dsigma_3 \int_Q  \dsigma_2 \int_Q\dsigma_1.
\end{split}
\end{equation*}
Since $Q\subseteq F_1\cap F_2 \cap F_3$, the cubes $F_i$ are ordered by inclusion. By symmetry, it suffices to consider the case $F_1\subseteq F_2\subseteq F_3$. Since $\pi(Q)=(F_1,F_2,F_3)$, we have that $F_i'\subsetneq Q$ for every $F'_i\in\ch_{\cf_i}(F_i)$ that intersects $Q$, for each $i=1,2,3$. Therefore $\pi_{\cf_2}(F_1)=F_2$, $\pi_{\cf_3}(F_1)=F_3$, and $\pi_{\cf_3}(F_2)=F_3$. We impose this condition by defining
$
\cf_1(F_2,F_3):=\{F_1\in\cf_1 : \pi_{\cf_2}(F_1)=F_2, \pi_{\cf_3}(F_1)=F_3\},$ and $\cf_2(F_3):=\{F_2\in\cf_2 : \pi_{\cf_3}(F_2)=F_3\}.$
Now,

\begin{equation*}
\begin{split}
&\sum_{\substack{(F_1, F_2, F_3)\in(\cf_1,\cf_2,\cf_3)\\ \cf_1\subseteq \cf_2\subseteq \cf_3}} \langle f_3 \rangle^{\sigma_3}_{F_3} \langle f_2 \rangle^{\sigma_2}_{F_2}  \langle f_1 \rangle^{\sigma_1}_{F_1} \sum_{\substack{Q\in\cd:\\ \pi(Q)=(F_1,F_2,F_3)}} \lambda_Q \int_Q  \dsigma_3 \int_Q  \dsigma_2 \int_Q\dsigma_1\\
&\leq \sum_{F_3\in \cf_3}  \langle f_3 \rangle^{\sigma_3}_{F_3} \times \sum_{F_2\in\cf_2(F_3)}\langle f_2 \rangle^{\sigma_2}_{F_2} \\
&\times \int  \sum_{F_1\in\cf_1(F_2,F_3)} \Big(\sum_{\substack{Q\in\cd:\\ \pi(Q)=(F_1,F_2,F_3)}}   \lambda_Q \int_Q  \dsigma_3 \int_Q  \dsigma_2 1_Q \Big) \langle f_1 \rangle^{\sigma_1}_{F_1} \dsigma_1.
\end{split}
\end{equation*}
For the innermost summation, we obtain
\begin{eqnarray*}
&&\int \sum_{F_1\in\cf_1(F_2,F_3)}\Big(  \sum_{\substack{Q\in\cd:\\ \pi(Q)=(F_1,F_2,F_3)}}   \lambda_Q \int_Q  \dsigma_3 \int_Q  \dsigma_2 1_Q \Big) \langle f_1 \rangle^{\sigma_1}_{F_1} \dsigma_1\\
&\leq &\int\sum_{F_1\in\cf_1(F_2,F_3)} \Big(  \sum_{\substack{Q\in\cd:\\ \pi(Q)=(F_1,F_2,F_3)}}   \lambda_Q \int_Q  \dsigma_3 \int_Q  \dsigma_2 1_Q \Big) \sup_{F'_1 \in \cf_1(F_2,F_3)} \langle f_1 \rangle^{\sigma_1}_{F_1'} 1_{F_1'} \dsigma_1\\
&\leq &\int \Big( \sum_{\substack{Q\in\cd:\\ Q\subseteq F_2}}   \lambda_Q \int_Q  \dsigma_3 \int_Q  \dsigma_2 1_Q \Big) \sup_{F'_1 \in \cf_1(F_2,F_3)} \langle f_1 \rangle^{\sigma_1}_{F_1'} 1_{F_1'} \dsigma_1\\
&\leq &\Nnorm{T_{F_2}(\sigma_2,\sigma_3)}_{L^{p_1'}(\sigma_1)} \BNorm{\sup_{F_1 \in \cf_1(F_2,F_3)} \langle f_1 \rangle^{\sigma_1}_{F_1} 1_{F_1}}{L^{p_1}(\sigma_1)}\\
&\leq &\Nnorm{T_{F_2}(\sigma_2,\sigma_3)}_{L^{p_1'}(\sigma_1)} \Big(\sum_{F_1 \in \cf_1(F_2,F_3)} \big(\langle f_1 \rangle^{\sigma_1}_{F_1}\big)^{p_1} \sigma_1(F_1)\Big)^{1/p_1}.
\end{eqnarray*}

Next, we use H\"older's inequality iteratively, from the inner summations to the outer ones.
 Recall that $r_3$ and $r$ are defined by
 \begin{equation*}
\begin{split}
  \frac{1}{{r_3}}:=\big(1-\frac{1}{p_1}-\frac{1}{p_2}\big)_+,\qquad \frac 1r:=\big(1-\frac{1}{p_1}-\frac{1}{p_2} -\frac{1}{p_3}\big)_+ .
\end{split}
\end{equation*}

By the first application of H\"older's inequality, we have
\begin{equation}\label{eq:innerHolder}
\begin{split}
&\sum_{F_2\in\cf_2(F_3)} \langle f_2 \rangle^{\sigma_2}_{F_2} \sigma_2(F_2)^{1/p_2}  \bigg(\frac{\Nnorm{T_{F_2}(\sigma_2,\sigma_3)}_{L^{p_1'}(\sigma_1)}}{\sigma_2(F_2)^{1/p_2}}\bigg) \\
&\qquad\times \Big(\sum_{F_1 \in \cf_1(F_2,F_3)} \big(\langle f_1 \rangle^{\sigma_1}_{F_1}\big)^{p_1} \sigma_1(F_1)\Big)^{1/p_1}\\
&\leq \bigg\|\bigg\{ \frac{\Nnorm{T_{F_2}(\sigma_2,\sigma_3)}_{L^{p_1'}(\sigma_1)}}{\sigma_2(F_2)^{1/p_2}} \bigg\}_{F_2\in\cf_2(F_3)} \bigg\|_{\ell^{r_3}} \\
&\qquad\times  \bigg(\sum_{F_2\in\cf_2(F_3)} (\langle f_2 \rangle^{\sigma_2}_{F_2} )^{p_2} \sigma_2(F_2)  \bigg)^{1/{p_2}}\\
&\qquad \times \Big(\sum_{\substack{F_2\in\cf_2(F_3) \\ F_1 \in \cf_1(F_2,F_3)}} \big(\langle f_1 \rangle^{\sigma_1}_{F_1}\big)^{p_1} \sigma_1(F_1)\Big)^{1/p_1}.
\end{split}
\end{equation}

By H\"older's inequality again, it follows that
\begin{eqnarray*}
&&\sum_{F_3\in \cf_3} \langle f_3 \rangle^{\sigma_3}_{F_3} \sigma_3(F_3)^{1/p_3}\times RHS\eqref{eq:innerHolder} \\
&\leq & \left\|\bigg\{ \bigg\|\bigg\{ \frac{\Nnorm{T_{F_2}(\sigma_2,\sigma_3)}_{L^{p_1'}(\sigma_1)}}{\sigma_2(F_2)^{1/p_2}\sigma_3(F_3)^{1/p_3}} \bigg\}_{F_2\in\cf_2 : F_2\subseteq F_3}\bigg\|_{\ell^{r_3}} \bigg\}_{F_3\in\cf_3}\right\|_{\ell^r} \\
&&\qquad\times\Big(\sum_{\substack{F_3\in\cf_3 \\ F_2\in\cf_2(F_3) \\ F_1 \in \cf_1(F_2,F_3)}}
   \big(\langle f_1 \rangle^{\sigma_1}_{F_1}\big)^{p_1} \sigma_1(F_1)\Big)^{1/p_1} \times \\
&&\qquad\times  \bigg(\sum_{\substack{F_3\in\cf_3 \\ F_2\in\cf_2(F_3)}} (\langle f_2 \rangle^{\sigma_2}_{F_2} )^{p_2} \sigma_2(F_2)  \bigg)^{1/{p_2}} \bigg(\sum_{F_3\in\cf_3 } (\langle f_3 \rangle^{\sigma_3}_{F_3} )^{p_3} \sigma_3(F_3)  \bigg)^{1/{p_3}} \\
&\leq&  \mathfrak T_{1,2,3}  \prod_{i=1}^3  \Big(\sum_{F_i \in \cf_i } \big(\langle f_i \rangle^{\sigma_i}_{F_i}\big)^{p_i} \sigma_i(F_i)\Big)^{1/p_i},
\end{eqnarray*}
where in the last step we used the facts that
$\sum_{F_3\in\cf_3} \sum_{\substack{F_2\in\cf_2:\\ \pi_{\cf_3}(F_2)=F_3}}=\sum_{F_2\in\cf_2}$ and
$$\sum_{F_3\in\cf_3} \sum_{\substack{F_2\in\cf_2:\\ \pi_{\cf_3}(F_2)=F_3}}\sum_{\substack{F_1\in\cf_1:\\ \pi_{\cf_2}(F_1)=F_2\\\pi_{\cf_3}(F_1)=F_3}}\leq \sum_{F_3\in\cf_3} \sum_{\substack{F_2\in\cf_2:\\ \pi_{\cf_3}(F_2)=F_3}}\sum_{\substack{F_1\in\cf_1:\\ \pi_{\cf_2}(F_1)=F_2}}= \sum_{F_1\in\cf_1}.$$
The proof is completed by Lemma~\ref{lm:carleson}, which implies that
$$
 \Big(\sum_{F_i \in \cf_i } \big(\langle f_i \rangle^{\sigma_i}_{F_i}\big)^{p_i} \sigma_i(F_i)\Big)^{1/p_i}  \leq 2p_i' \Nnorm{f_i}_{L^{p_i}(\sigma_i)}
$$
for each $i=1,2,3$.

\subsection{Discrete Wolff potential: the case $\frac 1{p_1}+\frac 1{p_2}+\frac 1{p_3}<1$}

In this subsection, we prove Theorem~\ref{thm:discreteWolff}. The reader is encouraged to recall the notation in the statement of that Theorem.

\begin{proof}[Proof of Theorem~\ref{thm:discreteWolff}]
First, we prove the `$\lesssim$' part.
By Theorem~\ref{thm:bilinear}, we only need to dominate the sequential testing constants by the
the discrete Wolff potentials.
By symmetry, we only need to consider the case $(i,j,k)=(1,2,3)$. Consider
\[
\left\|\bigg\{ \bigg\|\bigg\{ \frac{\Nnorm{T_{F_2}(\sigma_2,\sigma_3)}_{L^{p_1'}(\sigma_1)}}{\sigma_2(F_2)^{1/p_2}\sigma_3(F_3)^{1/p_3}} \bigg\}_{F_2\in\cf_2 : F_2\subseteq F_3}\bigg\|_{\ell^{r_3}} \bigg\}_{F_3\in\cf_3}\right\|_{\ell^r},
\]
where $\cf_2,\cf_3$ are $\sigma_2, \sigma_3$-sparse sequence, respectively. Using notation as in \eqref{eq:defLambdaQsigma}, we have
\begin{eqnarray*}
&&\bigg\|\bigg\{ \frac{\Nnorm{T_{F_2}(\sigma_2,\sigma_3)}_{L^{p_1'}(\sigma_1)}}{\sigma_2(F_2)^{1/p_2}} \bigg\}_{F_2\in\cf_2 : F_2\subseteq F_3}\bigg\|_{\ell^{r_3}}\\
&=&\bigg( \sum_{F_2\in \cf_2\atop F_2\subset F_3} \frac 1{\sigma_2(F_2)^{r_3/{p_2}}}
\bigg(\int\bigg(\sum_{Q\subset F_2} \lambda_Q \sigma_2(Q)\sigma_3(Q) 1_Q   \bigg)^{p_1'}   \dsigma_1       \bigg)^{r_3/{p_1'}}      \bigg)^{\frac 1{r_3}}\\
&\overset{(*)}{\eqsim}&\bigg( \sum_{F_2\in \cf_2\atop F_2\subset F_3} \frac 1{\sigma_2(F_2)^{r_3/{p_2}}}
\bigg(\sum_{Q\subset F_2} \lambda_Q \lambda_{Q,\sigma_1}^{p_1'-1}\sigma_1(Q)\sigma_2(Q)\sigma_3(Q)      \bigg)^{r_3/{p_1'}}      \bigg)^{\frac 1{r_3}}\\
&\lesssim&\bigg( \sum_{F_2\in \cf_2\atop F_2\subset F_3}
\bigg(  \frac 1{\sigma_2(F_2)}\int_{F_2}\sum_{Q\subset F_2} \lambda_Q \lambda_{Q,\sigma_1}^{p_1'-1}\sigma_1(Q)\sigma_3(Q)  1_Q  \dsigma_2    \bigg)^{r_3/{p_1'}}    \sigma_2(E_{\cf_2}(F_2))  \bigg)^{\frac 1{r_3}}\\
&\le& \bigg\|  M_{\mathcal D}^{\sigma_2}\Big(\sum_{Q\subset F_3} \lambda_Q \lambda_{Q,\sigma_1}^{p_1'-1}\sigma_1(Q)\sigma_3(Q)  1_Q\Big)           \bigg\|_{L^{r_3/{p_1'}}(\sigma_2)}^{1/{p_1'}}\\
&\lesssim& \bigg\| \sum_{Q\subset F_3} \lambda_Q \lambda_{Q,\sigma_1}^{p_1'-1}\sigma_1(Q)\sigma_3(Q)  1_Q        \bigg\|_{L^{r_3/{p_1'}}(\sigma_2)}^{1/{p_1'}},
\end{eqnarray*}
where in the step marked with $(*)$ we used Lemma \ref{lem:dyadicSum}.
Recalling also the notation $\lambda_{Q,\sigma_2,\sigma_1}$ from \eqref{eq:defLambdaQsigma}, and using Lemma~\ref{lem:dyadicSum} and the maximal function estimate again, we have
\begin{eqnarray*}
&&\left\|\bigg\{ \bigg\|\bigg\{ \frac{\Nnorm{T_{F_2}(\sigma_2,\sigma_3)}_{L^{p_1'}(\sigma_1)}}{\sigma_2(F_2)^{1/p_2}\sigma_3(F_3)^{1/p_3}} \bigg\}_{F_2\in\cf_2 : F_2\subseteq F_3}\bigg\|_{\ell^{r_3}} \bigg\}_{F_3\in\cf_3}\right\|_{\ell^r}\\
&=& \bigg( \sum_{F_3\in\cf_3}\frac 1{\sigma_3(F_3)^{r/{p_3}}}
\bigg(   \int_{F_3} \bigg(  \sum_{Q\subset F_3} \lambda_Q \lambda_{Q,\sigma_1}^{p_1'-1}\sigma_1(Q)\sigma_3(Q)  1_Q  \bigg)^{r_3/{p_1'}}\dsigma_2    \bigg)^{r/{r_3}}                                                                        \bigg)^{1/{r}}\\
&\eqsim& \bigg( \sum_{F_3\in\cf_3}\frac 1{\sigma_3(F_3)^{r/{p_3}}}
\bigg(    \sum_{Q\subset F_3} \lambda_Q \lambda_{Q,\sigma_1}^{p_1'-1}\lambda_{Q,\sigma_2,\sigma_1}^{\frac{r_3}{p_1'}-1}\sigma_1(Q)\sigma_2(Q)\sigma_3(Q) \bigg)^{r/{r_3}}   \bigg)^{1/{r}}\\
&\lesssim& \bigg\|  \bigg(  \sum_{Q\in\mathcal D} \lambda_Q \lambda_{Q,\sigma_1}^{p_1'-1}\lambda_{Q,\sigma_2,\sigma_1}^{\frac{r_3}{p_1'}-1}\sigma_1(Q)\sigma_2(Q)  1_Q\bigg)^{1/{r_3}}           \bigg\|_{L^{r}(\sigma_3)}\\
&=&\| \mathcal W_{\sigma_1, \sigma_2}[\sigma_3]^{1/{r_3}}\|_{L^{r}(\sigma_3)}.
\end{eqnarray*}

Next, we shall show that
\[
\| \mathcal W_{\sigma_1, \sigma_2 }[\sigma_3]^{1/{r_3}}\|_{L^{r}(\sigma_3)}\lesssim \| T(\,\cdot\,\sigma_1, \,\cdot\,\sigma_2)\|_{L^{p_1}(\sigma_1)\times L^{p_2}(\sigma_2)\rightarrow L^{p_3'}(\sigma_3)}=:\|T\|.
\]
By duality, we have
\begin{eqnarray*}
&&\| \mathcal W_{\sigma_1, \sigma_2}[\sigma_3]^{1/{r_3}}\|_{L^{r}(\sigma_3)}^{r_3}\\
&=& \sup\Big\{\sum_{Q\in\mathcal D} \lambda_Q \lambda_{Q,\sigma_1}^{p_1'-1}\lambda_{Q,\sigma_2,\sigma_1}^{\frac{r_3}{p_1'}-1}\sigma_1(Q)\sigma_2(Q)\int_Q g \dsigma_3:
  \|g\|_{L^{(r/{r_3})'(\sigma_3)}}=1\Big\}.
\end{eqnarray*}
In the following, we will suppress the supremum and give an uniform bound for the right side of the equality.
We have
\begin{eqnarray*}
&& \sum_{Q\in\mathcal D} \lambda_Q \lambda_{Q,\sigma_1}^{p_1'-1}\lambda_{Q,\sigma_2,\sigma_1}^{\frac{r_3}{p_1'}-1}\sigma_1(Q)\sigma_2(Q)\int_Q g \dsigma_3\\
&=&  \sum_{Q\in\mathcal D} \lambda_Q \lambda_{Q,\sigma_1}^{p_1'-1}\sigma_1(Q)\sigma_2(Q)\sigma_3(Q)\bigg(\frac{\int_Q g \dsigma_3}{\sigma(Q)}\bigg)^{p_1'/{r_3}}\\
&&\quad\times \bigg(\frac 1{\sigma_2(Q)}\bigg(\frac{\int_Q g \dsigma_3}{\sigma(Q)}\bigg)^{p_1'/{r_3}}\sum_{Q'\subset Q}\lambda_{Q'}\lambda_{Q',\sigma_1}^{p_1'-1}\sigma_1(Q')\sigma_2(Q')\sigma_3(Q')\bigg)^{\frac{r_3}{p_1'}-1}\\
&\le& \sum_{Q\in\mathcal D} \lambda_Q \lambda_{Q,\sigma_1}^{p_1'-1}\sigma_1(Q)\sigma_2(Q)\int_Q (M_{\mathcal D}^{\sigma_3}g)^{p_1'/{r_3}} \dsigma_3\\
&&\quad\times \bigg(\frac 1{\sigma_2(Q)}\sum_{Q'\subset Q}\lambda_{Q'}\lambda_{Q',\sigma_1}^{p_1'-1}\sigma_1(Q')\sigma_2(Q')\int_{Q'} (M_{\mathcal D}^{\sigma_3}g)^{p_1'/{r_3}} \dsigma_3\bigg)^{\frac{r_3}{p_1'}-1}\\
&\overset{(*)}{\eqsim}&\int  \bigg(  \sum_{Q\in\mathcal D} \lambda_Q \lambda_{Q,\sigma_1}^{p_1'-1}\sigma_1(Q)\int_Q (M_{\mathcal D}^{\sigma_3}g)^{p_1'/{r_3}} \dsigma_3    1_Q                                                     \bigg)^{r_3/{p_1'}} \dsigma_2,
\end{eqnarray*}
where, in the step marked with $(*)$, we used Lemma \ref{lem:dyadicSum}.

Again, we use duality. Let $h\ge0$ with $\|h\|_{L^{(r_3/{p_1'})'}(\sigma_2)}=1$. We then have
\begin{eqnarray*}
&&\sum_{Q\in\mathcal D} \lambda_Q \lambda_{Q,\sigma_1}^{p_1'-1}\sigma_1(Q)\int_Q (M_{\mathcal D}^{\sigma_3}g)^{p_1'/{r_3}} \dsigma_3   \int_Q h \dsigma_2\\
&=& \sum_{Q\in\mathcal D} \lambda_Q \sigma_1(Q)\sigma_2(Q)\sigma_3(Q)\bigg(\frac{\int_Q (M_{\mathcal D}^{\sigma_3}g)^{p_1'/{r_3}} \dsigma_3 }{\sigma_3(Q)}\bigg)^{1/{p_1'}} \bigg(\frac{ \int_Q h \dsigma_2}{\sigma_2(Q)}\bigg)^{1/{p_1'}}\\
&&\quad\times\bigg(\frac 1{\sigma_1(Q)}\bigg(\frac{\int_Q (M_{\mathcal D}^{\sigma_3}g)^{p_1'/{r_3}} \dsigma_3 }{\sigma_3(Q)}\bigg)^{1/{p_1'}} \bigg(\frac{ \int_Q h \dsigma_2}{\sigma_2(Q)}\bigg)^{1/{p_1'}}\\
&&\quad\times\sum_{Q'\subset Q}\lambda_{Q'}\sigma_1(Q')\sigma_2(Q')\sigma_3(Q')\bigg)^{p_1'-1}\\
&\le& \sum_{Q\in\mathcal D} \lambda_Q \sigma_1(Q)\int_Q M_{\mathcal D}^{\sigma_3}((M_{\mathcal D}^{\sigma_3}g)^{p_1'/{r_3}})^{1/{p_1'}} \dsigma_3\int_Q (M_{\mathcal D}^{\sigma_2}h)^{1/{p_1'}}\dsigma_2\\
&&\quad\times\bigg(\frac 1{\sigma_1(Q)}\sum_{Q'\subset Q}\lambda_{Q'}\sigma_1(Q')\int_{Q'} M_{\mathcal D}^{\sigma_3}((M_{\mathcal D}^{\sigma_3}g)^{p_1'/{r_3}})^{1/{p_1'}} \dsigma_3 \\
&&\quad\times\int_{Q'} (M_{\mathcal D}^{\sigma_2}h)^{1/{p_1'}}\dsigma_2\bigg)^{p_1'-1}\\
&\overset{(*)}{\eqsim}&\| T((M_{\mathcal D}^{\sigma_2}h)^{1/{p_1'}}\sigma_2, M_{\mathcal D}^{\sigma_3}((M_{\mathcal D}^{\sigma_3}g)^{p_1'/{r_3}})^{1/{p_1'}}\sigma_3)   \|_{L^{p_1'}(\sigma_1)}^{p_1'}\\
&\le&\|T\|^{p_1'} \cdot\|(M_{\mathcal D}^{\sigma_2}h)^{1/{p_1'}}\|_{L^{p_2}(\sigma_2)}^{p_1'}
\| M_{\mathcal D}^{\sigma_3}((M_{\mathcal D}^{\sigma_3}g)^{p_1'/{r_3}})^{1/{p_1'}}\|_{L^{p_3}(\sigma_3)}^{p_1'}\\
&\lesssim &\|T\|^{p_1'},
\end{eqnarray*}
where, in the step marked with $(*)$, we used Lemma \ref{lem:dyadicSum}, and, in the last step, the boundedness of the maximal operator together with the algebraic relations
$$
\frac{p_2}{p_1'}=\Big(\frac{r_3}{p_1'}\Big)', \text{ and } \frac{p_3}{r_3}=\Big(\frac{r}{r_3}\Big)'.
$$
Therefore,
\[
\sum_{Q\in\mathcal D} \lambda_Q \lambda_{Q,\sigma_1}^{p_1'-1}\lambda_{Q,\sigma_2,\sigma_1}^{\frac{r_3}{p_1'}-1}\sigma_1(Q)\sigma_2(Q)\int_Q g \dsigma_3\lesssim\|T\|^{r_3},
\]
and consequently,
\[
\| \mathcal W_{\sigma_1, \sigma_2 }[\sigma_3]^{1/{r_3}}\|_{L^{r}(\sigma_3)}\lesssim \|T\|.
\]
\end{proof}

\subsection{Abstract Wolff potential}
In this subsection, we also extend the abstract Wolff potential to the bilinear setting. We will give an analogous result to Proposition~\ref{prop:abstractWolff}. In this subsection, $T$ can be any positive bilinear operator. The corresponding localized operator $T_Q$ can be understood as $T(\cdot 1_Q , \cdot 1_Q  )$  or $ 1_Q T(\cdot 1_Q  , \cdot  1_Q )$.

Let $(i,j,k)$ be a permutation of $(1,2,3)$. We define the abstract Wolff potential functions $W$ and constants $\mathfrak{W}$ by cases depending on the exponents $p_i,p_j,p_k$:
\begin{itemize}
\item Case $\frac 1{p_i}+\frac 1{p_j}\ge 1$:
$$
\mathfrak{W}_{T, (i,j,k)}:=
\sup_{Q\in\cd}\frac{\|T_Q( \sigma_j,  \sigma_k)\|_{L^{p_i'}(\sigma_i)}}
{\sigma_j(Q)^{1/{p_j}}\sigma_k(Q)^{1/{p_k}}}.
$$
\end{itemize}
For the case $\frac 1{p_i}+\frac 1{p_j}< 1$, we define the auxiliary function
$$
W_{Q,(i,j,k)}:=\sup_{\substack{Q'\in\cd: \\Q'\subset Q}}\frac{1_{Q'}}{\sigma_j(Q')}\|T_{Q'}( \sigma_j,  \sigma_k)\|_{L^{p_i'}(\sigma_i)}^{p_i'}.
$$
\begin{itemize}
\item Case $\frac 1{p_i}+\frac 1{p_j}< 1\text{ and } \frac 1{p_i}+\frac 1{p_j}+\frac 1{p_k}\geq 1$:
$$
\mathfrak{W}_{T,(i,j,k)}:= \sup_{Q\in\cd} \frac{\| W_{Q,(i,j,k)}^{1/{p_i'}} \|_{L^{r_k}(\sigma_j)}}{\sigma_k(Q)^{1/{p_k}}} 
$$
\item Case $\frac 1{p_i}+\frac 1{p_j}< 1\text{ and } \frac 1{p_i}+\frac 1{p_j}+\frac 1{p_k}< 1$:
\begin{equation*}
\begin{split}
&W_{T,(i,j,k)}:=\sup_{Q\in\mathcal{D}}\frac{1_Q}{\sigma_k(Q)} \|W_{Q,(i,j,k)}^{1/p_i'}\|_{L^{r_k}(\sigma_j)}^{r_k}, \\
&\mathfrak{W}_{T,(i,j,k)}:= \|W_{T,(i,j,k)}^{1/r_k}\|_{L^r(\sigma_k)} .
\end{split}
\end{equation*}
\end{itemize}

\subsubsection{Case $\frac 1{p_1}+\frac 1{p_2}+\frac 1{p_3}\ge 1$}

\begin{theorem}\label{thm:ge1}
Suppose $1/{p_1}+1/{p_2}+1/{p_3}\ge 1$.
Let $T$ be a positive bilinear operator and $\mathfrak T_{i,j,k}$ and $\mathfrak{W}_{T,(i,j,k)}$ be defined as above. Then
\[
\mathfrak{W}_{T,(i,j,k)} \eqsim \mathfrak T_{i,j,k} \lesssim \|T\|.
\]
\end{theorem}

Combining Theorem~\ref{thm:bilinear} and Theorem~\ref{thm:ge1}, we immediately have the following:

\begin{corollary}
Suppose $1/{p_1}+1/{p_2}+1/{p_3}\ge 1$.
Let $\sigma_i$ be a locally finite Borel measures on $\mathbb R^n$ for $i=1,2,3$ and $T(\,\cdot\,\sigma_1, \,\cdot\,\sigma_2)$ be as in \eqref{eq:TandTQbil}.
Then
\[
\| T(\,\cdot\,\sigma_1, \,\cdot\,\sigma_2)\|_{L^{p_1}(\sigma_1)\times L^{p_2}(\sigma_2)\rightarrow L^{p_3'}(\sigma_3)}\eqsim \sum_{(i,j,k)\in S_3}
  \mathfrak{W}_{T,(i,j,k)}.
\]
\end{corollary}

\begin{proof}[Proof of Theorem~\ref{thm:ge1}]We note that $\mathfrak T_{i,j,k} \lesssim \|T\|$ is proven in Subsection \ref{subsec:nec-sequential} for a general positive bilinear operator $T$.

First, we prove that $\mathfrak{T}_{i,j,k}\lesssim \mathfrak{W}_{T,(i,j,k)}$.
If $\frac 1{p_i}+\frac 1{p_j}\ge 1$, then it is obvious that
\begin{eqnarray*}
\frac{\|T_{F_j}(\sigma_j,\sigma_k)\|_{L^{p_i'}(\sigma_i)}}{\sigma_j(F_j)^{1/p_j}\sigma_k(F_k)^{1/p_k}}
\le\frac{\|T_{F_j}(\sigma_j,\sigma_k)\|_{L^{p_i'}(\sigma_i)}}{\sigma_j(F_j)^{1/p_j}\sigma_k(F_j)^{1/p_k}}
\le \mathfrak{ W}_{T,(i,j,k)}.
\end{eqnarray*}
If $\frac 1{p_i}+\frac 1{p_j}<1$, then we have
\begin{equation*}
\begin{split}
 &\bigg\|\bigg\{ \frac{\Nnorm{T_{F_j}(\sigma_j,\sigma_k)}_{L^{p_i'}(\sigma_i)}}{\sigma_j(F_j)^{1/p_j}} \bigg\}_{F_j\in\cf_j : F_j\subseteq F_k}\bigg\|_{\ell^{r_k}}\\
&=\bigg(\sum_{F_j\in\cf_j\atop F_j\subset F_k}   \frac{\Nnorm{T_{F_j}(\sigma_j,\sigma_k)}_{L^{p_i'}(\sigma_i)}^{r_k}}{\sigma_j(F_j)^{r_k/{p_i'}}} \sigma_j(F_j)\bigg)^{1/{r_k}}
\lesssim \|W_{F_k,(i,j,k)}\|_{L^{r_k/{p_i'}}(\sigma_j)}^{1/{p_i'}}.
\end{split}
\end{equation*}
Since $r=\infty$, it follows that
\begin{equation*}
\begin{split}
&\left\|\bigg\{ \bigg\|\bigg\{ \frac{\|T_{F_j}(\sigma_j,\sigma_k)\|_{L^{p_i'}(\sigma_i)}}{\sigma_j(F_j)^{1/p_j}\sigma_k(F_k)^{1/p_k}} \bigg\}_{F_j\in\cf_j : F_j\subseteq F_k}\bigg\|_{\ell^{r_k}} \bigg\}_{F_k\in\cf_k}\right\|_{\ell^r}\\
&\lesssim \sup_{F_3\in\cf_3}\frac{\|W_{F_k,(i,j,k)}\|_{L^{r_k/{p_i'}}(\sigma_j)}^{1/{p_i'}}}{\sigma(F_k)^{1/{p_k}}}
\leq \mathfrak{W}_{T, (i,j,k)}.
\end{split}
\end{equation*}

Next, we prove that $\mathfrak{T}_{i,j,k}\gtrsim \mathfrak{W}_{T,(i,j,k)}$.
 If $\frac 1{p_i}+\frac 1{p_j}\ge 1$, it is obvious that $\mathfrak{W}_{T,(i,j,k)}\lesssim \mathfrak T_{i,j,k}$.
So we focus on the case $\frac 1{p_i}+\frac 1{p_j}< 1$. We shall show that
\[
\frac{\| W_{Q,(i,j,k)}^{1/{p_i'}} \|_{L^{r_k}(\sigma_j)}}{\sigma_k(Q)^{1/{p_k}}}\le C \mathfrak T_{i,j,k}
\]
for any fixed dyadic cube and the constant $C$ is independent of the choice of $Q$.  Of course, $\cf_k:=\{Q\}$ is $\sigma_k$-sparse.
Let $\cf_j$ be the stopping cubes starting at $Q$ and defined by
\begin{equation*}
  \operatorname{ch}_{\cf_j}(F):=\Big\{\text{maximal }F'\subsetneq F:\frac{\|T_{F'}( \sigma_j,  \sigma_k)\|_{L^{p_i'}(\sigma_i)}^{p_i'}}{\sigma_j(F')}
  >2\frac{\|T_{F}( \sigma_j,  \sigma_k)\|_{L^{p_i'}(\sigma_i)}^{p_i'}}{\sigma_j(F)}\Big\}.
\end{equation*}
Then it is easy to check that $\cf_j$ is $\sigma_j$-sparse and we have
\[
W_{Q,(i,j,k)}(x)\lesssim \sum_{F_j\in\cf_j}\frac{\|T_{F_j}( \sigma_j,  \sigma_k)\|_{L^{p_i'}(\sigma_i)}^{p_i'}}{\sigma_j(F_j)}  1_{F_j}.
\]
By Lemma \ref{lem:pythagoras},
\begin{equation*}
\begin{split}
&\frac{\| W_{Q,(i,j,k)}^{1/{p_i'}} \|_{L^{r_k}(\sigma_j)}}{\sigma_k(Q)^{1/{p_k}}}\\
&\lesssim \sigma_k(Q)^{-1/{p_k}}\Big( \sum_{F_j\in\cf_j}\frac{\|T_{F_j}( \sigma_j,  \sigma_k)\|_{L^{p_i'}(\sigma_i)}^{r_k}}{\sigma_j(F_j)^{r_k/{p_i'}}} \sigma(F_j)\Big)^{1/{r_k}}\\
&=\left\|\bigg\{ \bigg\|\bigg\{ \frac{\|T_{F_j}(\sigma_j,\sigma_k)\|_{L^{p_i'}(\sigma_i)}}{\sigma_j(F_j)^{1/p_j}\sigma_k(F_k)^{1/p_k}} \bigg\}_{F_j\in\cf_j : F_j\subseteq F_k}\bigg\|_{\ell^{r_k}} \bigg\}_{F_k\in\cf_k}\right\|_{\ell^r}
\lesssim\mathfrak T_{i,j,k}.
\end{split}
\end{equation*}
\end{proof}

\subsubsection{Case $\frac 1{p_1}+\frac 1{p_2}+\frac 1{p_3}<1$}


\begin{theorem}\label{thm:ge2}
Suppose $1/{p_1}+1/{p_2}+1/{p_3}<1$.
Let $T$ be a positive bilinear operator and $\tilde{\mathfrak T}_{i,j,k}$ and $W_{T,(i,j,k)}$ be defined as above. Then
\[
\mathfrak{W}_{{T,(i,j,k)}}:=\|W_{T,(i,j,k)}^{1/r_k}\|_{L^r(\sigma_k)}\eqsim   \tilde{\mathfrak T}_{i,j,k} \lesssim \| T\|.
\]
\end{theorem}

By combining Theorem~\ref{thm:bilinear} with Theorem \ref{thm:ge2}, we have:

\begin{corollary}
Suppose $1/{p_1}+1/{p_2}+1/{p_3}< 1$.
Let $\sigma_i$ be a locally finite Borel measurse on $\mathbb R^n$ for $i=1,2,3$ and $T(\,\cdot\,\sigma_1, \,\cdot\,\sigma_2)$ be as in \eqref{eq:TandTQbil}.
Then
\[
\| T(\,\cdot\,\sigma_1, \,\cdot\,\sigma_2)\|_{L^{p_1}(\sigma_1)\times L^{p_2}(\sigma_2)\rightarrow L^{p_3'}(\sigma_3)}
\eqsim \sum_{(i,j,k)\in S_3}\|W_{T,(i,j,k)}^{1/r_k}\|_{L^r(\sigma_k)}.
\]
\end{corollary}

\begin{proof}[Proof of Theorem~\ref{thm:ge2}]We note that $\tilde{\mathfrak T}_{i,j,k} \lesssim \|T\|$ is proven in Subsection \ref{subsec:nec-sequential} for a general positive bilinear operator $T$.

 By symmetry, we  just consider $(i,j,k)=(1,2,3)$. First, we prove that $ \tilde{\mathfrak T}_{1,2,3} \lesssim \|W_{T,(1,2,3)}^{1/r_3}\|_{L^r(\sigma_3)}$.We have
\begin{equation*}
\begin{split}
&\bigg\|\bigg\{ \frac{\Nnorm{T_{F_2}(\sigma_2,\sigma_3)}_{L^{p_1'}(\sigma_1)}}{\sigma_2(F_2)^{1/p_2}} \bigg\}_{F_2\in\cf_2(F_3)}\bigg\|_{\ell^{r_3}}\\
&=\bigg(\sum_{F_2\in\cf_2(F_3)}   \frac{\Nnorm{T_{F_2}(\sigma_2,\sigma_3)}_{L^{p_1'}(\sigma_1)}^{r_3}}{\sigma_2(F_2)^{r_3/{p_1'}}} \sigma_2(F_2)\bigg)^{1/{r_3}}
\lesssim \|W_{F_3,(1,2,3)}\|_{L^{r_3/{p_1'}}(\sigma_2)}^{1/{p_1'}}.
\end{split}
\end{equation*}
Therefore,
\begin{equation*}
\begin{split}
&\left\|\bigg\{ \bigg\|\bigg\{ \frac{\Nnorm{T_{F_2}(\sigma_2,\sigma_3)}_{L^{p_1'}(\sigma_1)}}{\sigma_2(F_2)^{1/p_2}\sigma_3(F_3)^{1/p_3}} \bigg\}_{F_2\in\cf_2(F_3)}\bigg\|_{\ell^{r_3}} \bigg\}_{F_3\in\cf_3}\right\|_{\ell^r}\\
&\lesssim \bigg( \sum_{F_3\in\cf_3} \frac{\|W_{F_3,(1,2,3)}\|_{L^{r_3/{p_1'}}(\sigma_2)}^{r/{p_1'}}}{\sigma(F_3)^{r/{p_3}}}    \bigg)^{1/r}\\
&= \bigg( \sum_{F_3\in\cf_3} \frac{\|W_{F_3,(1,2,3)}\|_{L^{r_3/{p_1'}}(\sigma_2)}^{r/{p_1'}}}{\sigma(F_3)^{r/{r_3}}}\sigma(F_3)    \bigg)^{1/r}
\lesssim \|  W_{T,(1,2,3)}^{1/{r_3}}  \|_{L^r(\sigma_3)}.
\end{split}
\end{equation*}

Next, we prove that $ \tilde{\mathfrak T}_{1,2,3} \gtrsim \|W_{T,(1,2,3)}^{1/r_3}\|_{L^r(\sigma_3)}$. For this, we may assume that the $\sup_{Q\in\mathcal{D}}$ in the definition of $W_{T,(i,j,j)}$ is replaced by $\sup_{Q\subseteq Q_N}$ for some big dyadic cube $Q_N$, as the original case then follows by monotone convergence as $Q_N$ exhausts all $\R^n$.

 Let $\cf_3$ be the stopping cubes starting at $Q_N$ and defined by
\begin{equation*}
  \operatorname{ch}_{\cf_3}(F):=\Big\{\text{maximal }F'\subsetneq F:\frac{\|W_{F',(1,2,3)}\|_{L^{r_3/{p_1'}}(\sigma_2)}^{r_3/{p_1'}}}{\sigma_3(F')}
  >2\frac{\|W_{F,(1,2,3)}\|_{L^{r_3/{p_1'}}(\sigma_2)}^{r_3/{p_1'}}}{\sigma_3(F)}\Big\}.
\end{equation*}
Then it is easy to check that $\cf_3$ is $\sigma_3$-sparse and we have
\[
W_{T,(1,2,3)}   \lesssim \sum_{F_3\in\cf_3}\frac{\|W_{F_3,(1,2,3)}\|_{L^{r_3/{p_1'}}(\sigma_2)}^{r_3/{p_1'}}}{\sigma_3(F_3)} 1_{F_3}
\]
By using Lemma \ref{lem:pythagoras},
\begin{eqnarray*}
\|  W_{T,(1,2,3)}^{1/{r_3}}  \|_{L^r(\sigma_3)}
&\lesssim& \Big( \sum_{F_3\in\cf_3}\frac{\|W_{F_3,(1,2,3)}\|_{L^{r_3/{p_1'}}(\sigma_2)}^{r/{p_1'}}}{\sigma_3(F_3)^{r/{r_3}}} \sigma(F_3) \Big)^{1/r}\\
&=& \Big( \sum_{F_3\in\cf_3}\frac{\|W_{F_3,(1,2,3)}\|_{L^{r_3/{p_1'}}(\sigma_2)}^{r/{p_1'}}}{\sigma_3(F_3)^{r/{p_3}}} \Big)^{1/r}.
\end{eqnarray*}
Next for each $F_3\in\cf_3$, we let $\cf_{2}(F_3)$ be the stopping cubes also starting at $F_3$ and defined by
\begin{equation*}
  \operatorname{ch}_{\cf_{2}(F_3)}(F):=\Big\{\text{maximal }F'\subsetneq F:\frac{\|T_{F'}( \sigma_2,  \sigma_3)\|_{L^{p_1'}(\sigma_1)}^{p_1'}}{\sigma_2(F')}
  >2\frac{\|T_{F}( \sigma_2,  \sigma_3)\|_{L^{p_1'}(\sigma_1)}^{p_1'}}{\sigma_2(F)}\Big\}.
\end{equation*}
Again, we know that $\cf_{2}(F_3)$ is $\sigma_2$-sparse. We have
\begin{eqnarray*}
W_{F_3,(1,2,3)}
&\le&
 \sum_{F_2\subset \cf_{2}(F_3)} \frac{\|T_{F_2}( \sigma_2,  \sigma_3)\|_{L^{p_1'}(\sigma_1)}^{p_1'}}{\sigma_2(F_2)} 1_{F_2}.
\end{eqnarray*}
By Lemma \ref{lem:pythagoras} again,
\begin{equation*}
\begin{split}
&\|  W_{T,(1,2,3)}^{1/{r_3}}  \|_{L^r(\sigma_3)}\\
&\lesssim \left\|\bigg\{ \bigg\|\bigg\{ \frac{\Nnorm{T_{F_2}(\sigma_2,\sigma_3)}_{L^{p_1'}(\sigma_1)}}{\sigma_2(F_2)^{1/p_2}\sigma_3(F_3)^{1/p_3}} \bigg\}_{F_2\in\cf_{2}(F_3)}\bigg\|_{\ell^{r_3}} \bigg\}_{F_3\in\cf_3}\right\|_{\ell^r}
\leq \tilde{\mathfrak T}_{1,2,3}.
\end{split}
\end{equation*}
This completes the proof.
\end{proof}

\section{Two weight testing condition for linearized maximal operators}
\subsection{Linear case}
Fix a (finite) collection $\cd$ of dyadic cubes. For non-negative real numbers $\lambda_Q$, the {\it maximal operator} $M^*(\,\cdot\,\sigma)$ is defined by
$$
M^*(f\sigma):=\sup_{Q\in\cd} \lambda_Q \int_Q f \dsigma 1_Q.
$$
Let $\mathcal{E}:=\{E(Q)\subseteq Q : Q\in\cd\}$ be a collection of pairwise disjoint sets.
The {\it linearized maximal operator} $M_{\mathcal{E}}(\,\cdot\,\sigma)$ and its localized version $M_{\mathcal{E},R}(\,\cdot\,\sigma)$ are defined by $$
M_{\mathcal{E}}(f\sigma):=\sum_Q \lambda_Q \int_Q f \dsigma\, 1_{E(Q)}\quad\text{and}\quad M_{\mathcal{E},R}(f\sigma):=\sum_{Q\subseteq R} \lambda_Q \int_Q f \dsigma\, 1_{E(Q)}.
$$
We notice that for each function $f$ there exists a collection $\mathcal{E}$ of pairwise disjoint sets $E(Q)\subseteq Q$ such that
$$
M^*(f\sigma)=M_{\mathcal{E}} (f\sigma);
$$
For example, we can choose
\begin{equation*}
\begin{split}
&E(Q):=\{x\in\br^d : M^*(f\sigma)(x)=\lambda_Q \int_Q f \dsigma, \text{ and }\\
&\phantom{E(Q):=\{x\in Q :} \lambda_{R} \int_R f \dsigma< \lambda_Q \int_Q f \dsigma \text{ whenever } R\supsetneq Q\}.
\end{split}
\end{equation*}
This implies:
\begin{lemma}We have
$$
\Norm{M^*(\cdot \sigma)}{L^q(\omega)\to L^p(\sigma)}=\sup_{\mathcal{E}} \Norm{M_\mathcal{E}(\,\cdot\,\sigma)}{L^q(\omega)\to L^p(\sigma)},
$$
where the supremum is over all collections $\mathcal{E}$ of pairwise disjoint sets $E(Q)\subseteq Q$.
\end{lemma}
The following theorem extends Sawyer's \cite[Theorem A]{sawyer1982} characterization of a two-weight norm inequality for maximal operators $M(\,\cdot\,\sigma):L^p(\sigma)\to L^q(w) $ in the case $q\geq  p$ to the case $q<p$.

\begin{theorem}\label{thm:maximalfunction}Let $\mathcal{E}$ be a collection of pairwise disjoint sets $E(Q)\subseteq Q$. For $p,q\in(1,\infty)$ and two measures $\sigma,\omega$, we have
\begin{equation*}
  \Norm{M_{\mathcal{E}}(\,\cdot\,\sigma)}{L^p(\sigma)\to L^q(\omega)}
  \eqsim_p \mathfrak{M}_{\mathcal{E},r},
\end{equation*}
where $r\in(1,\infty]$ is determined by $\displaystyle
  \frac{1}{r}=\Big(\frac{1}{q}-\frac{1}{p}\Big)_+$ and
\begin{equation*}
\begin{split}
  \mathfrak{M}_{\mathcal{E},r} & :=\sup_{\cf}\BNorm{\Big\{\frac{\Norm{M_{\mathcal{E},F}(\sigma)}{L^q(\omega)}}{\sigma(F)^{1/p}}\Big\}_{F\in\cf}}{\ell^r},
\end{split}
\end{equation*}
where the supremum is over all subcollections $\cf$ that are sparse with respect to $\sigma$.
\end{theorem}

\begin{proof}The necessity of sequential testing holds for all linear, positive operators by Corollary \ref{cor:necessityseq}. We next prove the sufficiency. In the proof we suppress the dependence on the collection $\mathcal{E}$.
\begin{equation*}
\begin{split}
\Norm{M(f\sigma)}{L^q(w)}^q&=\sum_{Q} \lambda_Q \big(\int_Q f \dsigma\big)^q\, w(E(Q))\\
&=\sum_{Q} \lambda_Q \big(\langle f \rangle_Q^\sigma\big)^q \sigma(Q)^q\, w(E(Q))\\
&=\sum_F \sum_{Q: \pi (Q)=F} \lambda_Q \big(\langle f \rangle_Q^\sigma\big)^q \sigma(Q)^q\, w(E(Q))\\
&\leq 2^q \sum_F \big(\langle f \rangle_F^\sigma\big)^q  \sum_{Q: \pi(Q)=F} \lambda_Q \sigma(Q)^q\, w(E(Q))\\
&\leq 2^q \sum_F \big(\langle f \rangle_F^\sigma \sigma(F)^{1/p}\big)^q  \Big(\frac{\Norm{M_F(\sigma)}{L^q(w)}}{\sigma(F)^{1/p}}\Big)^q
\end{split}
\end{equation*}
In the case $1<q<p<\infty$, applying H\"older's inequality with the exponents $(r/q)'=p/q$ and $r/q$ yields
\begin{equation*}
\begin{split}
&\sum_F \big(\langle f \rangle_F^\sigma \sigma(F)^{1/p})^q  \big(\frac{\Norm{M_F(\sigma)}{L^q(w)}}{\sigma(F)^{1/p}}\big)^q\\
&\leq \Big(\sum_F \big(\langle f \rangle_Q^\sigma\big)^p \sigma(F)\Big)^{q/p} \BNorm{\Big\{\frac{\Norm{M_F(\sigma)}{L^q(w)}}{\sigma(F)^{1/p}}\Big\}_{F\in\cf}}{\ell^r}^q.
\end{split}
\end{equation*}
In the case  $1<p<q<\infty$, applying the inequality $\Norm{\,\cdot\,}{\ell^{s_1}}\geq \Norm{\,\cdot\,}{\ell^{s_2}}$ for $0<s_1\leq s_2\leq\infty$ yields
\begin{equation*}
\begin{split}
&\sum_F \big((\langle f \rangle_F^\sigma )^p\sigma(F)\big)^{q/p}  \big(\frac{\Norm{M_F(\sigma)}{L^q(w)}}{\sigma(F)^{1/p}}\big)^q\\
&\leq \BNorm{\Big\{\frac{\Norm{M_F(\sigma)}{L^q(w)}}{\sigma(F)^{1/p}}\Big\}_{F\in\cf}}{\ell^\infty}^q \Big(\sum_F \big(\langle f \rangle_Q^\sigma\big)^p \sigma(F)\Big)^{q/p}.
\end{split}
\end{equation*}
The proof is completed by the special case of the Carleson embedding theorem, Lemma \ref{lm:carleson}.
\end{proof}
\subsection{Bilinear case}
Fix a (finite) collection $\cd$ of dyadic cubes. For non-negative real numbers $\lambda_Q$, we define the {\it bilinear maximal operator} $\mathcal M^*(\,\cdot\,\sigma_1, \,\cdot\,\sigma_2)$ by
$$
\mathcal M^*(f_1\sigma_1, f_2\sigma_2):=\sup_{Q\in\cd} \lambda_Q \int_Q f_1 \dsigma \int_Q f_2 \dsigma 1_Q.
$$
Analogous to the linear case, we can also define the collection $\mathcal E$ and the corresponding operator $\mathcal M _{\mathcal E}$ and its localized version $\mathcal M_{\mathcal E, R}$.
We prove the following result
\begin{theorem}
Let $\mathcal{E}$ be a collection of pairwise disjoint sets $E(Q)\subseteq Q$. Let $p_1,p_2, q\in(1,\infty)$ and
 $\sigma_1, \sigma_2,\omega$ be measures. We have
\begin{equation*}
  \Norm{M_{\mathcal{E}}(\,\cdot\,\sigma_1, \,\cdot\,\sigma_2)}{L^{p_1}(\sigma_1)\times L^{p_2}(\sigma_2)\to L^q(\omega)}
  \eqsim  \mathfrak{M}_{\mathcal{E},r_1,r}+ \mathfrak{M}_{\mathcal{E},r_2,r},
\end{equation*}
where $r_1, r_2, r\in(1,\infty]$ are determined by
\begin{eqnarray*}
  \frac 1 {r_i}&=& \Big( \frac 1 q- \frac 1 {p_i} \Big)_+,\qquad i=1,2,\\
  \frac{1}{r}&=&\Big(\frac{1}{q}-\frac{1}{p_1}-\frac 1{p_2}\Big)_+
\end{eqnarray*}
and
\begin{equation*}
\begin{split}
  \mathfrak{M}_{\mathcal{E},r_1, r} & :=\sup_{\cf_1, \cf_2}
   \Big\|\Big\{\Big\|\Big\{\frac{\|\mathcal M_{\mathcal E, F_1}(\sigma_1, \sigma_2)\|_{L^q(\omega)}}{\sigma_1(F_1)^{1/{p_1}}\sigma_2(F_2)^{1/{p_2}}}\Big\}_{F_1\subset F_2}\Big\|_{\ell^{r_1}}\Big\}_{F_2}\Big\|_{\ell^r},
\end{split}
\end{equation*}
where the supremum is over all $\sigma_1$-sparse subcollections $\cf_1$ and $\sigma_2$-sparse subcollections $\cf_2$ of the dyadic cubes $\mathcal{D}$.
\end{theorem}

\begin{proof}
The necessity follows from the arguments in Subsection~\ref{subsec:nec-sequential}. We focus only on the sufficiency.
We have
\begin{eqnarray*}
&&\|\mathcal M_{\mathcal E}(f_1 \sigma_1, f_2\sigma_2)\|_{L^q(\omega)}^q\\
&=& \sum_{Q} \lambda_Q^q \Big( \int_Q f_1\dsigma_1\int_Q f_2\dsigma_2\Big)^q \omega(E(Q))\\
&=& \sum_{F_1, F_2} \sum_{\pi(Q)=(F_1, F_2)}\lambda_Q^q (\langle f_1\rangle_Q^{\sigma_1}\langle f_2\rangle_Q^{\sigma_2})^q\sigma_1(Q)^q \sigma_2(Q)^q\omega(E(Q))\\
&\le&4^q \sum_{F_1, F_2} (\langle f_1\rangle_{F_1}^{\sigma_1}\langle f_2\rangle_{F_2}^{\sigma_2})^q\sum_{\pi(Q)=(F_1, F_2)} \lambda_Q^q \sigma_1(Q)^q \sigma_2(Q)^q\omega(E(Q))\\
&:=&4^q( \Sigma_{F_1\subset F_2}+ \Sigma_{F_2\subset F_1}).
\end{eqnarray*}
By symmetry, we only need to estimate the first term.
By similar arguments as that in Subsection~\ref{subsec:suf-sequential}, we know that $\pi_{\cf_2}(F_1)=F_2$.
Therefore,
\begin{eqnarray*}
\Sigma_{F_1\subset F_2}&\le& \sum_{F_2}(\langle f_2\rangle_{F_2}^{\sigma_2})^q \sum_{F_1: \pi_{\cf_2}(F_1)=F_2}
(\langle f_1\rangle_{F_1}^{\sigma_1})^q \|\mathcal M_{\mathcal E, F_1}(\sigma_1, \sigma_2)\|_{L^q(\omega)}^q\\
&\le& \sum_{F_2}(\langle f_2\rangle_{F_2}^{\sigma_2})^q \Big(\sum_{F_1: \pi_{\cf_2}(F_1)=F_2}(\langle f_1\rangle_{F_1}^{\sigma_1})^{p_1}\sigma(F_1)\Big)^{q/{p_1}}\\
&&\quad\times\Big\|\Big\{\frac{\|\mathcal M_{\mathcal E, F_1}(\sigma_1, \sigma_2)\|_{L^q(\omega)}}{\sigma_1(F_1)^{1/{p_1}}}\Big\}_{F_1\subset F_2}\Big\|_{\ell^{r_1}}^q\\
&\le& \Big(\sum_{F_2}(\langle f_2\rangle_{F_2}^{\sigma_2})^{p_2}\sigma_2(F_2)\Big)^{q/{p_2}}
 \Big(\sum_{F_2}\sum_{F_1: \pi_{\cf_2}(F_1)=F_2}(\langle f_1\rangle_{F_1}^{\sigma_1})^{p_1}\sigma(F_1)\Big)^{q/{p_1}}\\
 &&\quad\times
 \Big\|\Big\{\Big\|\Big\{\frac{\|\mathcal M_{\mathcal E, F_1}(\sigma_1, \sigma_2)\|_{L^q(\omega)}}{\sigma_1(F_1)^{1/{p_1}}\sigma_2(F_2)^{1/{p_2}}}\Big\}_{F_1\subset F_2}\Big\|_{\ell^{r_1}}\Big\}_{F_2}\Big\|_{\ell^r}^q\\
 &\le& \mathfrak M_{\mathcal E, r_1, r}^q \|f_1\|_{L^{p_1}(\sigma_1)}^q \|f_2\|_{L^{p_2}(\sigma_2)}^q.
\end{eqnarray*}
\end{proof}
\section{An application to the bilinear fractional integrals}
Recall that the bilinear fractional integral $\mathcal{I}_\alpha$, for $\alpha\in (0,2n)$,  is defined by
\begin{equation}
\label{def:bilinearfractional}\mathcal I_\alpha(f_1 \sigma_1, f_2\sigma_2)(x):=\int_{\mathbb R^{2n}}\frac{f_1(y_1)f_2(y_2)}{(|x-y_1|+|x-y_2|)^{2n-\alpha}}\dsigma_1(y_1)\dsigma_2(y_2).
\end{equation}
In \cite{li2013}, the third author and Sun showed that
\begin{equation}\label{eq:pointwise}
  \mathcal I_\alpha(f_1 \sigma_1, f_2\sigma_2)(x)\eqsim
  \sum_{t\in\{0, 1/3\}^n}\mathcal I_\alpha^{\mathcal D^t}(f_1\sigma_1, f_2\sigma_2)(x),
\end{equation}
where
\[
\mathcal I_\alpha^{\mathcal D^t}(f_1\sigma_1, f_2\sigma_2)(x):=\sum_{Q\in\mathcal D^t}
\prod_{i=1}^2\frac 1{|Q|^{1-\alpha/{2n}}}\int_{Q}f_i\dsigma_i  1_Q(x)
\]
and, for each $t\in \{0, 1/3\}^n$,
\begin{equation*}\label{def:shifteddyadicsystem}
\mathcal D^t:=\{2^{-k}([0,1)^n+m+(-1)^k t): k\in \mathbb Z, m\in\mathbb Z^n\}.
\end{equation*}
We also define the localized operator $\mathcal I_{\alpha, Q}$ to be $\mathcal I_\alpha(\cdot 1_Q , \cdot 1_Q  )$
 or $ 1_Q \mathcal I_\alpha(\cdot 1_Q  , \cdot  1_Q )$.

For each permutation $(i,j,k)$ of $(1,2,3)$, the sequential testing constant $\mathfrak{T}^{\cd^t}_{\mathcal{I}_\alpha,(i,j,k)}$ and  the abstract Wolff potential constant $\mathfrak{W}^{\cd^t}_{\mathcal{I}_\alpha,(i,j,k)}$ are defined as before,  expect that we now use severel different dyadic systems $\cd^t$ and indicate this by the superscript ``${\cd^t}$'' in the notation.

Next, we recall the definitions for reader's convenience. The constant $\mathfrak{T}^{\cd^t}_{\mathcal{I}_\alpha,(i,j,k)}$ is defined by
\[
\mathfrak T_{\mathcal{I}_\alpha,{(i,j,k)}}^{\cd^t}:=\sup_{\cf_j, \cf_k}\left\|\bigg\{ \bigg\|\bigg\{ \frac{\|\mathcal I_{\alpha, F_j}(\sigma_j,\sigma_k)\|_{L^{p_i'}(\sigma_i)}}{\sigma_j(F_j)^{1/p_j}\sigma_k(F_k)^{1/p_k}} \bigg\}_{F_j\in\cf_j : F_j\subseteq F_k}\bigg\|_{\ell^{r_k}} \bigg\}_{F_k\in\cf_k}\right\|_{\ell^r},
\]
where $\cf_j, \cf_k$ are $\sigma_j, \sigma_k$-sparse subcollections of the dyadic system $\mathcal D^t$, respectively. The constant $\mathfrak{W}^{\cd^t}_{\mathcal{I}_\alpha,(i,j,k)}$ is defined by cases:
\begin{itemize}
\item Case $\frac 1{p_i}+\frac 1{p_j}\ge 1$:
$$
\mathfrak{W}^{\cd^t}_{\mathcal I_\alpha, (i,j,k)}:=
\sup_{Q\in\cd^t}\frac{\|\mathcal I_{\alpha, Q}( \sigma_j,  \sigma_k)\|_{L^{p_i'}(\sigma_i)}}
{\sigma_j(Q)^{1/{p_j}}\sigma_k(Q)^{1/{p_k}}}.$$
\end{itemize}
For the case $\frac 1{p_i}+\frac 1{p_j}< 1$, we define the auxiliary function
$$
 W^{\cd^t}_{Q,(i,j,k)}:=\sup_{\substack{Q'\in\cd^t:\\Q'\subseteq Q}}\frac{1_{Q'}}{\sigma_j(Q')}\|\mathcal I_{\alpha, Q'}( \sigma_j,  \sigma_k)\|_{L^{p_i'}(\sigma_i)}^{p_i'}.
$$
\begin{itemize}
\item Case $\frac 1{p_i}+\frac 1{p_j}< 1\text{ and } \frac 1{p_i}+\frac 1{p_j}+\frac 1{p_k}\geq 1$:
$$
\mathfrak{W}^{\cd^t}_{\mathcal{I}_\alpha,(i,j,k)}:= \sup_{Q\in\cd^t} \frac{\| (W^{\cd^t}_{Q,(i,j,k)})^{1/{p_i'}} \|_{L^{r_k}(\sigma_j)}}{\sigma_k(Q)^{1/{p_k}}} 
$$
\item Case $\frac 1{p_i}+\frac 1{p_j}< 1\text{ and } \frac 1{p_i}+\frac 1{p_j}+\frac 1{p_k}< 1$:
\begin{equation*}
\begin{split}
&   W^{\cd^t}_{\mathcal I_\alpha,(i,j,k)}:=\sup_{Q\in\cd^t}\frac{1_Q}{\sigma_k(Q)}\|(W^{\cd^t}_{Q,(i,j,k)})^{1/{p_i'}}\|_{L^{r_k}(\sigma_j)}^{r_k}, \\
&\mathfrak{W}^{\cd^t}_{\mathcal{I}_\alpha,(i,j,k)}:= \|(W^{\cd^t}_{\mathcal{I}_\alpha,(i,j,k)})^{1/r_k}\|_{L^r(\sigma_k)} .
\end{split}
\end{equation*}
\end{itemize}
\begin{theorem}Let $\alpha\in(0,2n)$. Let $\mathcal{I_\alpha}$ be the bilinear fractional integral defined in \eqref{def:bilinearfractional}. For $1<p_1,p_2,p_3<\infty$, and locally finite Borel measures $\sigma_1,\sigma_2,\sigma_3$ on $\br^n$, we have
\begin{equation*}
\begin{split}
&\|\mathcal I_\alpha (\,\cdot\,\sigma_1, \,\cdot\,\sigma_2)\|_{L^{p_1}(\sigma_1)\times L^{p_2}(\sigma_2)\rightarrow L^{p_3'}(\sigma_3)}\\
&\eqsim  \sum_{t\in\{0,1/3\}^n}\sum_{(i,j,k)\in S_3}
  \mathfrak{W}^{\cd^t}_{\mathcal I_\alpha,(i,j,k)}\\
  &\eqsim \sum_{t\in \{0, 1/3\}^n}\sum_{(i,j,k)\in S_3} \mathfrak T_{\mathcal{I}_\alpha,{(i,j,k)}}^{\cd^t}.
\end{split}
\end{equation*}

\end{theorem}
\begin{proof}
From the pointwise equivalence \eqref{eq:pointwise}, Theorem \ref{thm:bilinear} applied to the dyadic operators $\mathcal{I}^{\cd^t}_\alpha$, and the pointwise equivalence \eqref{eq:pointwise} again, it follows that
\begin{equation*}
\begin{split}
\| \mathcal{I}_\alpha \| &\eqsim \sum_{t\in\{0,1/3\}^n}\| \mathcal{I}^{\cd^t}_\alpha \| \eqsim \sum_{t\in\{0,1/3\}^n} \sum_{(i,j,k)\in S_3}
  \mathfrak{T}^{\cd^t}_{\mathcal I^{\cd^t}_\alpha,(i,j,k)}\\
  &\lesssim \sum_{t\in\{0,1/3\}^n} \sum_{(i,j,k)\in S_3}
  \mathfrak{T}^{\cd^t}_{\mathcal I_\alpha,(i,j,k)}.
  \end{split}
\end{equation*}
where $\mathfrak{T}^{\cd^t}_{\mathcal I_\alpha,(i,j,k)}$ denotes the sequential testing constant of the operator $\mathcal I_\alpha$ with respect the dyadic system $\cd^t$. By Theorem \ref{thm:ge1} and Theorem \ref{thm:ge2} together with the fact that $\mathfrak{T}^{\cd^t}_{\mathcal I_\alpha,(i,j,k)}\leq \tilde{\mathfrak{T}}^{\cd^t}_{\mathcal I_\alpha,(i,j,k)}$, we have that $\mathfrak{T}^{\cd^t}_{\mathcal I_\alpha,(i,j,k)}\lesssim \mathfrak{W}^{\cd^t}_{\mathcal I_\alpha,(i,j,k)}\lesssim \| \mathcal{I}_\alpha \|$, which completes the proof.
\end{proof}


\bibliography{weighted}

\begin{thebibliography}{10}

\bibitem{cov2004}
Carme Cascante, Joaquin~M. Ortega, and Igor~E. Verbitsky.
\newblock Nonlinear potentials and two weight trace inequalities for general
  dyadic and radial kernels.
\newblock {\em Indiana Univ. Math. J.}, 53(3):845--882, 2004.

\bibitem{cov2006}
Carme Cascante, Joaquin~M. Ortega, and Igor~E. Verbitsky.
\newblock On {$L^p-L^q$} trace inequalities.
\newblock {\em J. Lond. Math. Soc.}, 74(2):497--511, 2006.

\bibitem{hanninenhytonen2014}
Timo~S. H\"anninen and Tuomas~P. Hyt\"onen.
\newblock Operator-valued dyadic shifts and the {$T(1)$} theorem. {P}reprint.
\newblock 2014.
\newblock arXiv:1412.0470 [math.CA].

\bibitem{hytonen2012a}
Tuomas~P. Hyt\"onen.
\newblock The ${A}_2$ theorem: remarks and complements.
\newblock In {\em Harmonic Analysis and Partial Differential Equations}, volume
  612 of {\em Contemporary Mathematics}, pages 91--106. American Mathematical
  Society, Providence, RI, 2014.
\newblock arXiv:1212.3840 [math.CA].

\bibitem{lacey2009}
Michael~T. Lacey, Eric~T. Sawyer, and Ignacio Uriarte-Tuero.
\newblock {T}wo weight inequalities for discrete positive operators.
  {P}reprint.
\newblock 2009.
\newblock arXiv:0911.3437 [math.CA].

\bibitem{liSun2013}
Kangwei Li and Wenchang Sun.
\newblock Characterization of a two weight inequality for multilinear
  fractional maximal operators.
\newblock {\em Houston J. Math.}, 2013.
\newblock to appear, arXiv:1305.4267 [math.CA].

\bibitem{li2013}
Kangwei Li and Wenchang Sun.
\newblock Two weight norm inequalities for the bilinear fractional integrals.
  {P}reprint.
\newblock 2013.
\newblock arXiv:1312.7707 [math.CA].

\bibitem{NTV:1999}
F.~Nazarov, S.~Treil, and A.~Volberg.
\newblock The {B}ellman functions and two-weight inequalities for {H}aar
  multipliers.
\newblock {\em J. Amer. Math. Soc.}, 12(4):909--928, 1999.

\bibitem{sawyer1982}
Eric~T. Sawyer.
\newblock A characterization of a two-weight norm inequality for maximal
  operators.
\newblock {\em Studia Math.}, 75(1):1--11, 1982.

\bibitem{sawyer1984}
Eric~T. Sawyer.
\newblock Weighted {L}ebesgue and {L}orentz norm inequalities for the {H}ardy
  operator.
\newblock {\em Trans. Amer. Math. Soc.}, 281(1):329--337, 1984.

\bibitem{sawyer1988}
Eric~T. Sawyer.
\newblock A characterization of two weight norm inequalities for fractional and
  {P}oisson integrals.
\newblock {\em Trans. Amer. Math. Soc.}, 308(2):533--545, 1988.

\bibitem{sawyer1996}
Eric~T. Sawyer, Richard~L. Wheeden, and Shiying Zhao.
\newblock Weighted norm inequalities for operators of potential type and
  fractional maximal functions.
\newblock {\em Potential Anal.}, 5(6):523--580, 1996.

\bibitem{tanaka2014}
Hitoshi Tanaka.
\newblock A characterization of two-weight trace inequalities for positive
  dyadic operators in the upper triangle case.
\newblock {\em Potential Anal.}, 41(2):487--499, 2014.

\bibitem{tanaka2014b}
Hitoshi Tanaka.
\newblock The trilinear embedding theorem. {P}reprint.
\newblock 2014.
\newblock arXiv:1404.2694 [math.CA].

\bibitem{Vuorinen}
Emil Vuorinen.
\newblock {$L^p(\mu)\to L^q(\nu)$} characterization for well localized
  operators. {P}reprint.
\newblock 2014.
\newblock arXiv:1412.2127 [math.CA].

\end{thebibliography}
\bibliographystyle{plain}
\end{document}